\documentclass[a4paper,10pt]{article}

\usepackage{times}
\usepackage[T1]{fontenc}
\usepackage[latin1]{inputenc}
\usepackage{amsmath,amssymb}
\usepackage{amsthm}
\usepackage{amscd}
\usepackage{graphicx}
\usepackage[toc,page,titletoc]{appendix}

\usepackage{caption}
\usepackage{subcaption}

\newtheorem{theorem}{Theorem}[section]
\newtheorem{lemma}[theorem]{Lemma}

\newtheorem{proposition}[theorem]{Proposition}

\theoremstyle{definition}

\newtheorem{example}[theorem]{Example}
\newtheorem{remark}[theorem]{Remark}

\numberwithin{table}{section}
\numberwithin{equation}{section}

\begin{document}
\title{Multicentric calculus and the Riesz projection} 
\author{ Diana Apetrei and Olavi Nevanlinna }
\maketitle

\begin{center}
{\footnotesize\em 
Aalto University\\
Department of Mathematics and Systems Analysis\\
P.O.Box 11100\\
Otakaari 1M, Espoo\\
FI-00076 Aalto, Finland\\
 email:\\ Diana.Apetrei@aalto.fi\\Olavi.Nevanlinna\symbol{'100}aalto.fi\\[3pt]
}
\end{center}

\begin{abstract}
In multicentric holomorphic calculus one  represents the function $\varphi$  using a new polynomial variable $w=p(z)$ in such a way that when it is evaluated at the operator $A,$ then  $p(A)$ is small in norm. Usually it is assumed that $p$ has distinct roots. In this paper we discuss two related problems, the separation of a compact set (such as the spectrum) into different components by a polynomial lemniscate, respectively the application of the Calculus to the computation and the estimation of the Riesz spectral projection.  It may then become desirable the use of $p(z)^n$ as a new variable. We also develop the necessary modifications to incorporate the multiplicities in the roots.  
 
\end{abstract}

\bigskip
{\footnotesize\em
{ keywords:}  multicentric calculus, Riesz projections, spectral projections, sign function of an operator, lemniscates.  

}

\vspace{1cm}
This  is  a preprint of article \cite{l}, to appear, and  it is essentially equivalent with the article, but contains in appendix material which is not included in the article.
\section{Introduction}

Let $p(z)$ be a polynomial of degree $d$ with distinct roots $\lambda_1, \dots, \lambda_d$.  In  multicentric holomorphic calculus the polynomial is taken as a new variable $w=p(z)$ and functions $\varphi(z)$ are represented with the help of a vector-valued function $f$, mapping $w\mapsto f(w) \in \mathbb C^d$, \cite{a}.  For example, sets bounded by  lemniscates $|p(z)| = \rho$ are then mapped onto discs $|w| \le \rho$ and for $\rho$  small, $f$ has a rapidly converging Taylor series. 

The multicentric representation then yields a functional calculus  for operators (or matrices) $A$, if one has  found a polynomial $p$ such that $\|p(A)\|$ is small.   In fact, denote 
$$V_p(A)=\{ z \in \mathbb C \ : \ |p(z)| \le \|p(A)\| \}
$$
and observe that, by spectral mapping theorem,  the spectrum  $\sigma(A)$ satisfies 
$\sigma(A) \subset V_p(A)$.  If $f$ has a rapidly  converging Taylor series for $|w|  \le \|p(A)\|$, then $\varphi(A)$ can be written down by a  rapidly converging explicit series expansion.

Since $V_p(A)$ can have several components, one can define $\varphi=1$ in the neighborhood of some components, while $\varphi=0$ in a neighborhood of the others.
Then  $\varphi(A)$ represents the spectral projection onto the invariant subspace corresponding to the part of the spectrum where $\varphi=1$. 
The spectral projection
satisfies
$$
\varphi(A)=\frac{1}{2\pi i} \int_\gamma  (\lambda I-A)^{-1} d\lambda ,
$$
where $\gamma$ surrounds the appropriate componens of the spectrum, but the computational approach does not need the evaluation of the contour integral.
The coefficients for the Taylor series of $f$ can be computed with explicit recursion from those of $\varphi$ at the {\it local centers} $\lambda_j$ \cite{a}.  The approach also  yields a bound for $\| \varphi(A)\|$ by a generalization of the von Neumann theorem for contractions, see \cite{b}.   If the scalar function $\varphi$ is not holomorphic at the spectrum, the multicentric calculus leads to a new functional calculus to deal with, e.g. nontrivial Jordan blocks \cite{k}.

In this paper we take a closer look at the computation of spectral projections in   finding the stable and unstable invariant subspaces of an operator $A$.   The direct approach would be to ask for a polynomial $p$ with distinct roots such that 
$$
V_p(A) \cap i\mathbb R = \emptyset
$$
and then apply the calculus with $\varphi=1$ for $\textnormal{Re } z>0,$ respectively  $\varphi=0$ for $\textnormal{Re } z<0$.   However, we discuss this as two different subiects, one being the separation of the spectrum and the other being the computation of the projection. 

In order to discuss the separation, we denote
$$
V(p,\rho) = \{z \in \mathbb C  \ : |p(z) | \leq \rho\}
$$
and we let $K=\sigma(A),$ then we ask what is the minimal degree of a polynomial such that 
\begin{equation}\label{ekakysymys}
K \subset V(p,\rho)  \text { and }  V(p,\rho) \cap i\mathbb R = \emptyset.
\end{equation}
holds. By Hilbert's lemniscate theorem, see e.g. \cite{i},  such a polynomial with a minimal degree always exists. We model this question by  considering in place of $\sigma(A)$  two lines parallel to the imaginary axis as follows
$$
K= \{z=x+iy \ : \ x\in \{-1,1\}, \  |y| \le \tan(\alpha)\}
$$
and derive a sample of polynomials for which (\ref{ekakysymys}) holds when $\alpha$ grows. For $\alpha <\pi/4,$ the  minimal degree is  clearly $2$ but in general we are not able to prove exact lower bounds. 

Whenever  (\ref{ekakysymys}) holds, then the series expansion of $f$ representing $\varphi=1$ for $\textnormal{Re } z>0$ and $\varphi=0$ for $\textnormal{Re } z<0$ converges  in  $p(K),$ and if $\sigma(A)\subset K$ then we do obtain a convergent expresion for the projection onto to unstable invariant subspace.  However, whenever $p(A)$ would be nonnormal,  it could happen that  $V_p(A)  \cap i\mathbb R \not=\emptyset$, and then we would not get a bound for the projection by the generalization of von Neumann theorem.  To overcome this, note that from the spectral radius formula $r(B)= \lim \|B^n\|^{1/n}$  in such a case, there does exist an integer $n$ such that  with $q=p^n$ we have
$
V_q(A)\cap i\mathbb R =\emptyset$.

This leads to the other topic discussed in this paper.  With $q=p^n,$ the new variable there are no longer simple zeros and we shall therefore  derive the multicentric representations needed in this case.  This is done in Section $2$  together  while the model problem is discussed in Section $3$.   In Section $4$ we describe how to compute the series expansions when working with different monic polynomials and in the end we discuss a nonnormal small dimensional problem with $q(z)=(z^2-1)^n$.
 
\section{Representations and main estimates}

\subsection{Formulas and estimates}

 We need the basic formula for expressing a given function $\varphi(z)$ as a linear combination of functions $f_{j,k}(w^n)$ when $w=p(z),$ for $n$ a given positive integer.
 
 Let $p(z)$ be the monic polynomial of degree $d$ with distinct roots $\lambda_1, \dots ,\lambda_d.$ We denote by $\delta_k\in\mathbb{P}_{d-1}$ the Lagrange interpolation basis polynomials at $\lambda_j$
 
 $$ \delta_k(\lambda)=\frac{1}{p'(\lambda_k)} \prod_{j\neq k} (\lambda-\lambda_j). $$

Then the multicentric representation of $\varphi$ takes the form 
\begin{equation}
 \varphi(z)= \sum_{j=1}^{d} \delta_j(z)f_j(w)  \label{phiA},   \ \ \text {where  } \ \  w=p(z)
\end{equation}
and $f_j$'s are obtained from $ \varphi$ with the formula \cite{b} 
\begin{equation}\label{kaanteiskaava}
f_j(w) = \sum_{l=1}^d \delta_l(\lambda_j,w)\varphi(\zeta_l(w)),
\end{equation}
where  
$\zeta_l(w)$ denote the roots of $ p(\lambda)-w=0 $ and 
$$
\delta_l(\lambda,w)=\frac{p(\lambda)-w}{p'(\zeta_l(w))(\lambda-\zeta_l(w))}.
$$ 
 
When $\varphi$ is holomorphic, $f$ can also be computed from the  Taylor coefficients of $\varphi$ at the local centers $\lambda_j$.   In fact, 
$$
f_j(w)= \sum_{n=0}^\infty \frac{1}{n!} f_j^{(n)}(0) w^n,
$$
where $f_j^{(n)}(0)$ can be computed  recursively:\\ \\
$
\noindent (p'(\lambda_j))^nf_j^{(n)}(0)= \nonumber
$
\begin{equation}\label{ihanoikein}
 \varphi^{(n)}(\lambda_j) - \sum_{k=1}^d \sum_{m=0}^{n-1} 
{n\choose m} \delta_k^{(n-m)}(\lambda_j) \sum_{l=0}^m b_{m l}(\lambda_j)
f_k^{(l)} (0) -  \sum_{l=0}^{n-1} b_{nl}(\lambda_j)f_j^{(l)}(0).
\end{equation}
Here the  polynomials $b_{nm}$ are determined by
$$
b_{n+1,m}= b_{n,m-1} p' + b_{nm}'
$$
with
$b_{n0}=0, b_{1,1}=p'$ and $b_{nm}=0$ for $m>n$, see Proposition 4.3 in \cite{a}.

{\bf Remark}
Erratum: the last term on the right of (\ref{ihanoikein})  is missing from the formula (4.2) of \cite{a}).

Since the computations for $f_j$'s are done with power series, we can move to $p(z)^{n}=w^{n},$ because the expansions are done for that variable. Therefore we formulate the next theorem.

 \begin{theorem}\label{Thm2.1}
Suppose $p$ has simple zeros and assume  $\varphi$ is holomorphic in a neighborhood of $V(p,\rho)=\{z\in \mathbb C \ : \ |p(z)|\le \rho\}$ and given in the form
$$
\varphi(z) = \sum_{j=1}^d \delta_j(z) f_j(w),  \ \ \text { where } w=p(z).
$$
  Then  
$$
\varphi(z) = \sum_{j=1}^d \delta_j(z) [f_{j,0}(w^n) + \cdots + w^{n-1} f_{j, n-1}(w^n)],
$$
where  $f_{j,k}$  are  holomorphic in a neighborhood  of the disc $|w| \le \rho$ and given by
$$
 w^k f_{j,k}(w^n) =
  \frac{1}{n } \{ f_j(w)+ e^{-2\pi ik/n}f_j(e^{2\pi i/n}w) +\dots 
+e^{-2\pi i(n-1)k/n} f_j(e^{2\pi i(n-1)/n} w)\}.
$$

\end{theorem}

In order to prove this, we consider a fixed function $f_j$ and put $g=f_j.$ Then we set 
$$ w^k g_k(w^{n})= \frac{1}{n } \{ g(w)+ e^{-2\pi ik/n}g(e^{2\pi i/n}w) +\dots 
+e^{-2\pi i(n-1)k/n} g(e^{2\pi i(n-1)/n} w)\} 
$$ pointwise.

\begin{proposition} \label{Prop2.2}
Given an arbritary $n\in\mathbb{N}$ and the functions $g_i, \text{ } i=0,\dots ,n-1,$ for all $w \in \mathbb{C}$ we have:
$$
g(w) = g_0(w^n) + w g_1(w^n) + \dots + w^{n-1} g_{n-1}(w^n).
$$
\end{proposition}

\begin{proof}
Using the above formula of $g_k(w^{n}),$ for $k=0,1,\dots ,n-1,$ we have
\begin{eqnarray}
g_0(w^{n}) &=& \frac{1}{n} \{ g(w)+g(e^{2\pi i/n}w)+ g(e^{4\pi i/n}w)+\dots + g(e^{2\pi i (n-1)/n}w)\} \nonumber\\
w g_1(w^{n}) &=& \frac{1}{n} \{ g(w)+e^{-2\pi i/n}g(e^{2\pi i/n}w)+e^{-4\pi i/n}g(e^{4\pi i/n}w) +\dots \nonumber\\
&\text{ }& +e^{-2\pi i (n-1)/n} g(e^{2\pi i (n-1)/n}w)\}\nonumber\\
w^{2} g_2(w^{n}) &=& \frac{1}{n} \{ g(w)+e^{-4\pi i/n}g(e^{2\pi i/n}w)+ e^{-8\pi i/n}g(e^{4\pi i/n}w)+\dots \nonumber\\
&\text{ }& +e^{-4\pi i (n-1)/n} g(e^{2\pi i (n-1)/n}w)\}\nonumber\\
&\dots & \nonumber\\
w^{n-1} g_{n-1}(w^{n}) &=& \frac{1}{n} \{ g(w)+e^{-2\pi i(n-1)/n}g(e^{2\pi i/n}w)+e^{-4\pi i (n-1)/n} g(e^{4\pi i/n}w)\nonumber\\
&\text{ }&+\dots +e^{-2\pi i (n-1)^{2}/n} g(e^{2\pi i (n-1)/n}w)\}\nonumber.
\end{eqnarray}
Summing up all the terms we get
\begin{eqnarray}
&\text{}&\frac{1}{n} n g(w) + \frac{1}{n} g(e^{2\pi i/n}w) \sum_{k=0}^{n-1} e^{-2\pi ik/n} + \frac{1}{n} g(e^{4\pi i/n}w) \sum_{k=0}^{n-1} e^{-4\pi ik/n} + \dots \nonumber \\
&\text{}&  +\frac{1}{n} g(e^{2\pi i(n-2)/n}w)\sum_{k=0}^{n-1} e^{-2\pi i(n-2)k/n}+ \frac{1}{n} g(e^{2\pi i(n-1)/n}w)\sum_{k=0}^{n-1} e^{-2\pi i(n-1)k/n} \nonumber\\
&\text{}&= g(w)\nonumber
\end{eqnarray}
since all the other terms sum up to zero.
\end{proof}

\begin{proof}[Proof of Theorem \ref{Thm2.1}]
It follows now immediately from Proposition \ref{Prop2.2} together with the following proposition.
\end{proof}

\begin{proposition}
If $f_j$ is given for $|p(z)|\leq\rho,$ then $f_{j,k},$ $k=1,\dots, n-1,$ are defined for $|p(z)^{n}|\leq\rho^{n}$ and
\begin{equation}\label{fj}
f_j(p(z))= \sum_{k=0}^{n-1} p(z)^{k} f_{j,k}(p(z)^{n}), \quad \text{for }
|p(z)|\leq \rho. 
\end{equation}
Further, if $f_j(p(z))$ is analytic for $|p(z)|\leq\rho$ then so are $f_{j,k}(p(z)^{n}).$
\end{proposition}

\begin{proof}
First part is proved in Proposition \ref{Prop2.2}.

For the second part, we assume $f_j$ analytic, thus it can be written as a power series
\begin{equation}
f_j(p(z))= \sum_{m=0}^{\infty} \alpha_m p(z)^{m}. \label{fj_series}
\end{equation}

We know that pointwise we have
\begin{eqnarray}
 p(z)^k f_{j,k}(p(z)^{n})&=& \frac{1}{n } \left[ f_j(p(z))+ e^{-2\pi ik/n}f_j(e^{2\pi i/n}p(z)) +\dots \right. \nonumber\\
 &\text{ }& \left. +e^{-2\pi i(n-1)k/n} f_j(e^{2\pi i(n-1)/n} p(z))\right]. \label{wkfjk}
\end{eqnarray}

When we substitute (\ref{fj_series}) in (\ref{wkfjk}), we get
\begin{align*}
f_{j,0}(p(z)^{n}) =& \frac{1}{n} \left[ \alpha_0 + \alpha_1 p(z) +\alpha_2 p(z)^{2}+ \dots + \alpha_n p(z)^{n}+ \dots \right. \nonumber\\
&  +\alpha_0 +e^{2\pi i/n} \alpha_1 p(z)+ e^{4\pi i/n} \alpha_2 p(z)^{2} + \dots + \alpha_n p(z)^{n}+\dots \nonumber\\
&+ \dots \nonumber\\
&+\left.  \alpha_0 + e^{2\pi i(n-1)/n} \alpha_1 p(z)+ e^{4\pi i(n-1)/n} \alpha_2 p(z)^{2} + \dots + \alpha_n p(z)^{n}+\dots \right]. \nonumber
\end{align*}

Thus
$ \displaystyle 
 f_{j,0}(p(z)^{n})= \alpha_0 + \alpha_n p(z)^{n}+ \alpha_{2n} p(z)^{2n}+\dots,
$
since all the other terms vanish. 

We continue with $p(z)f_{j,1}(p(z)^{n}),$ so we get 
\begin{align*}
p(z)f_{j,1}(p(z)^{n}) =& \frac{1}{n} \left[ \alpha_0 + \alpha_1 p(z) +\alpha_2 p(z)^{2}+ \dots + \alpha_n p(z)^{n}+ \dots \right. \nonumber\\
&+  e^{-2\pi i/n}\alpha_0 + \alpha_1 p(z)+ e^{2\pi i/n} \alpha_2 p(z)^{2} \nonumber\\
&\quad + \dots + e^{2\pi i(n-1)/n}\alpha_n p(z)^{n}+ \alpha_{n+1} p(z)^{n+1} + \dots \nonumber\\
&+ \dots \nonumber\\
&+  e^{-2\pi i(n-1)/n}\alpha_0 + \alpha_1 p(z)+ e^{2\pi i(n-1)/n} \alpha_2 p(z)^{2} \nonumber\\
&\quad\left. + \dots +e^{2\pi i(n-1)^{2}/n} \alpha_n p(z)^{n}+\alpha_{n+1}p(z)^{n+1}+\dots \right]. \nonumber
\end{align*}

Therefore $\displaystyle
 f_{j,1}(p(z)^{n})= \alpha_1 p(z) + \alpha_{n+1} p(z)^{n+1}+ \alpha_{2n+1} p(z)^{2n+1}+\dots .$

\noindent In a similar way it follows that 
\begin{equation}
p(z)^{k}f_{j,k}(p(z)^{n})= \alpha_k p(z)^{k} + \alpha_{n+k} p(z)^{n+k}+ \alpha_{2n+k} p(z)^{2n+k}+\dots \label{wkfjk_series}
\end{equation}

Because all the coefficients $\alpha_{m_k},$ for $m_k=k,n+k,2n+k,3n+k,\dots ,$ come from $f_j$ which is analytic, we know that 
$ \displaystyle \limsup |\alpha_{m_k}|^{1/m_k}\leq \frac{1}{\rho},$
therefore we have that $p(z)^{k}f_{j,k}(p(z)^{n}) $ is analytic, i.e. a converging power series. 

Now, if we factor (\ref{wkfjk_series})
$$ p(z)^{k}f_{j,k}(p(z)^{n})=  p(z)^{k} \left[ \alpha_k  + \alpha_{n+k} p(z)^{n}+ \alpha_{2n+k} p(z)^{2n}+\dots \right] $$
we have that
$$
f_{j,k}(p(z)^{n})= \alpha_k  + \alpha_{n+k} p(z)^{n}+ \alpha_{2n+k} p(z)^{2n}+\dots $$
is a converging power series, thus is analytic for $|p(z)^{n}|\leq\rho^{n}$.

\end{proof}

\noindent Next we need to be able to bound $\varphi$ in terms of $f_{j,k}$'s and vice versa. The first one is straightforward.

\begin{proposition}\label{Prop2.3}
Denote $L(\rho)= \sup_{|p(z)|\le \rho} \sum_{j=1}^d |\delta_j(z)|.$  Then
$$
\sup_{|p(z)|\leq \rho} |\varphi (z)| \leq L(\rho) \max_{1\le j\le d} \sum_{k=0}^{n-1}  \sup_{|w| \le \rho}  |w^{k}f_{j,k}(w^{n})|.
$$
\end{proposition}

\noindent The other direction is more involved and we formulate it in the following theorem.

\begin{theorem}\label{Thm2.4}
Assume $p$ is a monic polynomial of degree $d$ with distinct roots and  $s(\rho)$ denotes the distance from  the lemniscate $|p(z)|=\rho$ to the nearest critical point of $p,$ $z_c,$  such that $|p(z_c)|>\rho$ . Then there exists a constant  $C$, depending on $p$ but not on $\rho,$ such that if $\varphi$ is  holomorphic  inside and in  a neighborhood of the lemniscate, then  each $f_{j,k}$ is holomorphic  for $|w| \le \rho$ and for $|w| \le \rho$, $1 \le j \le d$, $0\le k \le n-1,$ we have
\begin{equation}\label{jfa_bound}
\  |w^{k}f_{j,k} (w^n)| \ \le (1+ \frac{C}{s^{d-1}} ) \sup_{|p(z)| \le \rho}|\varphi(z)|.
\end{equation}

\end{theorem}

\noindent For the proof of this statement we need some lemmas. The aim is to bound the representing functions $f_{j,k}$ in terms of the original function $\varphi$.  To that end we first quote the basic result of  bounding $f_j$ in terms of $\varphi$ and then proceed bounding $f_{j,k}$ in terms of $f_j$. 

\begin{lemma}(Theorem 1.1 in \cite{b})\label{jfa}
Suppose $\varphi$ is holomorphic in a neighborhood of the set $\{ \zeta :   
|p(\zeta)|\le   \rho \}$  and let $s$ be as in Theorem \ref{Thm2.4}.   There exists a constant $C$,  depending on $p$ but independent of $\rho$ and $\varphi$,  such that 
$$
\sup_{|w| \le \rho}|f_j(w)| \le (1+ \frac{C}{s^{d-1}} ) \sup_{|p(z)| \le \rho}|\varphi(z)|, \qquad \text{for all } j=1,\dots ,d.
$$
\end{lemma}

\begin{lemma}\label{Lemma2.7}
In  the notation above we have
$$
\sup_{|w| \le \rho} |w^{k}g_k(w^n)| \le   \sup_{|w|\le \rho} |g(w)|.
$$

 \end{lemma}

\begin{proof}
We have 
$$ 
w^k g_k(w^{n})= \frac{1}{n } \{ g(w)+ e^{-2\pi ik/n}g(e^{2\pi i/n}w) +\dots 
+e^{-2\pi i(n-1)k/n} g(e^{2\pi i(n-1)/n} w)\}.
$$
The bound follows by taking the absolute values termwise.
  \end{proof}

\begin{proof}[Proof of Theorem \ref{Thm2.4}]
The proof follows immediately from these two lemmas.
\end{proof}

\begin{remark} \label{rem2.5}  Gauss-Lucas theorem asserts that given a polynomial $p$ with complex coefficients, all zeros of $p'$ belong to the convex hull of the set of zeros of $p,$ see \cite{g}. Thus, as soon as $\rho$ is large enough, all the critical points will stay inside the lemniscate whenever the lemniscate is just a single Jordan curve. But when we want to make the separation, we start squeezing the level of the lemniscate and this results in leaving at least one critical point outside. Therefore the $s$ we are measuring is the distance to the boundary from that particular critical point.

If the critical points  of $p$ are simple (as they generically are)  then the  dependency of the distance is inverse proportional and the coefficient in the theorem takes the form
$$
(1+ \frac{C}{s} ).
$$

\end{remark}

To see how one can find these contants $C$ and what they describe we will present an example with the computations of $\delta_l(\lambda_j,w)$ for $p(z)=z^{4}+1.$ These computations will be used in computing the constants. We also apply them for the Riesz projections. 

\begin{example} Let $p(z)=z^4+1$ with roots $\lambda_1=(-1)^{1/4},$ $\lambda_2=(-1)^{3/4},$ $\lambda_3=-(-1)^{1/4}$ and $\lambda_4=-(-1)^{3/4}.$ Let $w=z^4+1.$ Using the formula for $\delta_l(\lambda_j,w)$ we have:
\begin{eqnarray}
\delta_1(\lambda_1,w)&=& 1+\tfrac{3}{8}w+\tfrac{19}{64}w^2+\tfrac{33}{128}w^3+\dots \nonumber\\
\delta_1(\lambda_2,w)&=& -\tfrac{1+i}{8}w-\tfrac{3+4 i}{32}w^2-\tfrac{20+31 i}{256}w^3+\dots \nonumber\\
\delta_1(\lambda_3,w)&=& -\tfrac{1}{8}w-\tfrac{7}{64}w^2-\tfrac{13}{128}w^3+\dots \nonumber\\
\delta_1(\lambda_4,w)&=& -\tfrac{1-i}{8}w-\tfrac{3-4 i}{32}w^2-\tfrac{20-31 i}{256}w^3+\dots \nonumber\\
\delta_2(\lambda_1,w)&=& -\tfrac{1}{8}w-\tfrac{7}{64}w^2-\tfrac{13}{128}w^3+\dots \nonumber\\
\delta_2(\lambda_2,w)&=& -\tfrac{1-i}{8}w-\tfrac{3-4 i}{32}w^2-\tfrac{20-31 i}{256}w^3+\dots \nonumber\\
\delta_2(\lambda_3,w)&=& 1+\tfrac{3}{8}w+\tfrac{19}{64}w^2+\tfrac{33}{128}w^3+\dots \nonumber\\
\delta_2(\lambda_4,w)&=& -\tfrac{1+i}{8}w-\tfrac{3+4 i}{32}w^2-\tfrac{20+31 i}{256}w^3+\dots \nonumber\\
\delta_3(\lambda_1,w)&=& -\tfrac{1-i}{8}w-\tfrac{3-4 i}{32}w^2-\tfrac{20-31 i}{256}w^3+\dots \nonumber\\
\delta_3(\lambda_2,w)&=& 1+\tfrac{3}{8}w+\tfrac{19}{64}w^2+\tfrac{33}{128}w^3+\dots \nonumber\\
\delta_3(\lambda_3,w)&=& -\tfrac{1+i}{8}w-\tfrac{3+4 i}{32}w^2-\tfrac{20+31 i}{256}w^3+\dots \nonumber\\
\delta_3(\lambda_4,w)&=& -\tfrac{1}{8}w-\tfrac{7}{64}w^2-\tfrac{13}{128}w^3+\dots \nonumber\\
\delta_4(\lambda_1,w)&=& -\tfrac{1+i}{8}w-\tfrac{3+4 i}{32}w^2-\tfrac{20+31 i}{256}w^3+\dots \nonumber\\
\delta_4(\lambda_2,w)&=& -\tfrac{1}{8}w-\tfrac{7}{64}w^2-\tfrac{13}{128}w^3+\dots \nonumber\\
\delta_4(\lambda_3,w)&=& -\tfrac{1-i}{8}w-\tfrac{3-4 i}{32}w^2-\tfrac{20-31 i}{256}w^3+\dots \nonumber\\
\delta_4(\lambda_4,w)&=& 1+\tfrac{3}{8}w+\tfrac{19}{64}w^2+\tfrac{33}{128}w^3+\dots \nonumber
\end{eqnarray}

\noindent where $\zeta_l(w),$ for $l=\overline{1,4},$ are given by $\zeta_1(w)=(w-1)^{1/4},$ $\zeta_2(w)=-(w-1)^{1/4},$ $\zeta_3(w)=i(w-1)^{1/4}$ and $\zeta_4(w)=-i(w-1)^{1/4}.$
\end{example}

\begin{remark}  \label{ConstantC}From Lemma \ref{jfa} we see that
\begin{equation}
|\delta_l^{(m)}(\lambda_k,w)|  = {\sim}  \frac{C_m}{s^{m+1}}.\label{constC}
\end{equation}
where $m$ is the multiplicity of the nearest critical point of $p$ outside the lemniscate. 
\end{remark}

To be able to compute the constants $C$ from equation (\ref{constC}) for polynomials of degree $d\geq 4,$ we need to separate the spectrum by perturbing the roots with an angle $\varepsilon$ small enough  and by dropping the magnitude $\rho$ below $1.$ The perturbations for polynomials of degree $4,6,8,10,12$ and $14$ are described in the next section. 

Note that before the perturbation  we have multiple critical points, all inside the lemniscate, except for $0,$ which is on the level curve. After we perturb the roots, one critical point is left outside the lemniscate, while all the others remain inside. Therefore, in our case, the multiplicity $m$ is zero.

From now on we will choose a random value for $\varepsilon$ to make some experimental computations. All the computations below will work properly for any other random value  $\varepsilon$ small enough that ensures the desired separation. 

Now, if we choose a random perturbation with, for example, $\varepsilon=\pi/70,$ we get the following values for the constant $C$

\begin{table}[ht]
\caption{The constant C} \vspace{0.25 cm} 
\centering
\begin{tabular}{l | c c c c c c}
\hline\hline 
Degree & 4 & 6 & 8 & 10 & 12 & 14 \\ [0.5ex]
\hline
Constant C & 576.4344  & 1.4665  & 8.0721  & 2.2754 & 12.8520 & 4.0475 \\
\hline
\end{tabular}
\label{table:constantC}
\end{table}

The computations for $\sum |\delta_l(\lambda_k,w)|$ were made with Mathematica, see Appendix \ref{AppendixA}, and the values for $s,$ the smallest distance from the lemniscate to the nearest critical point, were computed with the help of Tiina Vesanen that provided a Matlab program, see Appendix \ref{AppendixB}.
For all these computations one has to choose a value for the level $\rho$, smaller than $1$. If one chooses level $\rho=1,$ then the value for $s$ will be zero, since the lemniscate passes through origin. Therefore we have choosen the minimum value for the level $\rho$ such that the lemniscate separates only in two parts when having a perturbation with $\varepsilon=\pi/70.$ Hence we have registered the following data

\begin{table}[ht]
\centering
\begin{tabular}{c | c c c c c c}
\hline\hline 
Degree & 4 & 6 & 8 & 10 & 12 & 14 \\ [0.5ex]
\hline
 $\sum |\delta_l(\lambda_k,w)|$ & 2293.81  & 6.5122  & 46.6599  & 16.3586 & 83.4547 & 16.7046 \\
\hline
$s$ & 0.2513 & 0.2252 & 0.1730 & 0.1391 & 0.1540 & 0.2423 \\
\hline
\end{tabular}
\end{table}

From the table of constants $C$ one can see that the lemniscate bifurcates differently even with a small $\varepsilon.$

\begin{remark} \label{Deg2_C} For the quadratic polynomial $p(z)=z^{2}-1,$ one does not need to perturb the roots, but just to decrease the magnitude of $\rho$ below $1.$ In this case, for example, if the level is $\rho= 0.9$ then one gets $C=0.2,$ while if the level is $\rho = 0.99$ then $C=0.6956 .$
\end{remark}

\subsection{Application to Riesz projection}

Let $A$ be a bounded operator in a Hilbert space such that
$$
V_{p^n}(A)=\{z \ :  \  |p(z)^{n}| \le \|p(A)^n\| \}.
$$
In order to compute the Riesz projection we take $\varphi=1$ in one of the components and $\varphi=0$ in the other components of $V_{p^{n}}(A)$
and 

\begin{equation}
f_j(p(z))=\sum_{k=0}^{n-1} p(z)^{k} f_{j,k}(p(z)^{n}). \nonumber
\end{equation} 

We shall apply the following theorem, if $\varphi$ is holomorphic in a neighbourhood of the unit disk $\mathbb{D}$ and $A\in \mathcal{B}(H),$ then
\begin{equation}
\|\varphi (A)\| \leq \sup_{\mathbb{D}} |\varphi|, \label{star}
\end{equation} 
see e.g. \cite{h}.

From Theorem \ref{Thm2.4} we have for $|p(z)|\leq\rho$  that
\begin{equation}
|p(z)^{k}f_{j,k}(p(z)^{n})|\leq (1+\frac{C}{s^{d-1}})\sup_{|p(z)|\leq \rho} |\varphi(z)|, \label{jfk_bound}
\end{equation}

\noindent and since $f_{j,k}$ are analytic, we can apply (\ref{star}) to each of them, so with $\rho=\|p(A)^{n}\|^{1/n}$
\begin{equation}
\| f_{j,k}(p(A)^{n})\| \leq \sup_{|p(z)|\leq \rho} |f_{j,k}(p(z)^{n})|. \label{vonN}
\end{equation}

We have,
$$ \| p(A)^{k} f_{j,k} (p(A)^{n}) \| \leq  \| p(A)^{k}\| \| f_{j,k} (p(A)^{n})\| $$
and from (\ref{vonN}) we get
$$ \| p(A)^{k} f_{j,k} (p(A)^{n}) \| \leq \| p(A)^{k}\| \sup_{|p(z)|\leq \rho} |f_{j,k}(p(z)^{n})| .$$

By the Maximum principle we see that
$$ \| p(A)^{k}\| \sup_{|p(z)|\leq \rho} |f_{j,k}(p(z)^{n})| =  \| p(A)^{k}\| \rho^{-k} \sup_{|p(z)|\leq \rho} | p(z)^{k} f_{j,k}(p(z)^{n})|. $$

Therefore, by (\ref{jfk_bound}), we have
$$
 \| p(A)^{k} f_{j,k} (p(A)^{n}) \| \leq \| (p(A)/\rho)^{k} \| \left( 1+\frac{C}{s^{d-1}}\right) \sup_{|p(z)|\leq \rho} |\varphi(z)|. 
$$

Substituting now in the decomposition (\ref{phiA}), $\varphi(A)$ becomes the Riesz projection and it is bounded by
\begin{equation}\label{RieszBound}
\| \varphi(A) \| \leq \left[ \left( 1+\frac{C}{s^{d-1}}\right) \frac{\| p(A)^{k} \|}{\rho^{k}} \sum_{j=1}^{d} \| \delta_j(A)\| \right] \sup_{|p(z)|\leq \rho} |\varphi(z)|. 
\end{equation}

Thus we have proven the following theorem.

\begin{theorem}\label{ThmRiesz}
Given a polynomial $p$ of degree $d,$ with distinct roots and a bounded operator $A$ in a Hilbert space, we asume that the "expression"
$$V_{p^{n}}(A)=\{ z \ :  \  |p(z)^{n}| \le \|p(A)^n\| \} $$
has at least two components. Set $\rho=\|p(A)^{n}\|^{1/n}$ and let $\varphi$ be a function such that $\varphi=1$ in one component of $V_{p^{n}}(A)$ and $\varphi=0$ in the others. Let $s$ be the distance from the nearest outside critical point to the boundary of $V_{p^{n}}(A)$.

Then, considering $f_{j,k}$ as given by Theorem \ref{Thm2.1}, one has
$$ \varphi(A) = \sum_{j=1}^{d}\delta_j(A)\sum_{k=0}^{n-1}p(A)^{k}f_{j,k}(p(A)^{n}),$$
which is the Riesz spectral projection onto the invariant subspace corresponding to the spectrum inside the component where $\varphi=1.$

The bound for the norm of $\varphi(A)$ is given by (\ref{RieszBound}).
 
\end{theorem}

\begin{remark}\label{Rem2.14}
In applications we will consider $\varphi=1$ in the components of $V_{p^{n}}(A)$ that are on the right complex half-plane and $\varphi=-1$ in the components on the left complex half-plane. This way the computations are more symmetrical. In this situation there is only one critical point at the origin, which is simple, so we will have $\displaystyle \left(1+\frac{C}{s}\right)$ in the formula for the bound of the Riesz projection.

\end{remark}

\section{Separating polynomials}

\subsection{Separation tasks}

In this section we shall discuss separting issues by lemniscates.  To that end,  given a polynomial
$p$
denote by $V(p, \rho)$ the set
$$
V(p, \rho)= \{z \in \mathbb C \ : |p(z)| \le \rho \}. 
$$
A key result in this context is the following.  Let $K\subset \mathbb C$ be compact and  such that $\mathbb C \setminus K$ is connected. For $\delta>0$ denote further
$$
K(\delta) = \{z \ : \ \text{dist}( z, K) < \delta\}.
$$
Then there exists a polynomial $p$ and $\rho >0$ such that
$$ 
K \subset V(p,\rho) \subset K(\delta).
$$
In particular, if $K=K_1 \cup K_2$ and $K_1 \cap K_2= \emptyset$ since $K$ is compact, then  for small enough $\delta $
$$
K_1(\delta) \cap K_2(\delta)=\emptyset
$$
as well.  Thus, $V(p,\rho)$  separates the components $K_1$ and $K_2$ respectively.
Suppose  that we have two analytic functions $\varphi_j$, each analytic in $K_j(\delta)$.
We  can view them as just one analytic function 
$$\varphi:   K(\delta) \rightarrow \mathbb C ,
$$
 where $\varphi$ agrees to $\varphi_j$ on $K_j(\delta)$.   We are interested in particular in the case where $\varphi_j$ is constant.  Multicentric  representation  then gives a power series which is simultaneously valid in both components.
 
So, we can ask, for such a separtion task  what is the smallest degree of a polynomial achieving this.  

We model this as follows:
Let 
$$
K_1(\delta) = \{z=1 +iy \ : |y| \le \tan(\delta)  \}
$$
and $K_2(\delta)$ symmetrically on the other side of the imaginary axis:
$$
K_2(\delta) = \{z=-1 +iy \ : |y| \le \tan(\delta)  \}
$$

Our first problem concerns the minimal degree of a polynomial $p$ such that
$$
K_j(\delta) \subset V(p, \rho)
$$ 
and 
$$
V(p,\rho) \cap i\mathbb R = \emptyset.
$$
This is discussed in the next  subsection.

Another natural task is  related to existence of a logarithm.  So, $C$ is again compact and we assume $0 \notin C$.   Now there exists a single valued logarithm in $C$ if and only if $0$ is in the unbounded component of the complement of $C$. That is,  the set $C$ does not separate  origin from infinity. The natural task here is to find a polynomial $p$ such that
$$
C \subset V(p,\rho)
$$
and
$$
0\notin V(p,\rho).
$$
As $V(p,\rho)$ is polynomially convex, this suffices for  representing the logarithm in $V(p,\rho)$. 

\subsection{ Model problem}

Let $p(z)=z^d-1$ be a monic polynomial of complex variable $z,$ with $|p(z)|=1.$ The polynomial $p$ can be written as
$$ p(z)=\prod_{j=1}^{d}(z-e^{i \theta_j}) $$
where $e^{i\theta_j}$ are the roots of $p$ and $\theta_j$ are the angles of the roots. For $d=4m,$ $m\geq 1$ we have

$$ p(z)=\prod_{j=1}^{m} (z-e^{i \theta_j})\prod_{j=1}^{m} (z-e^{-i \theta_j})\prod_{j=1}^{m} (z+e^{i \theta_j})\prod_{j=1}^{m} (z+e^{-i \theta_j}). $$

We are interested to separate the spectrum of $p,$ as discussed earlier in this paper. In this sense, we first check if there are roots laying on the imaginary axis. If so, a rotation with angle $\pi /d$ is applied so that no roots touch $i\mathbb{R}.$ 

Next we perturb the roots as follows: the four roots that are closest to the imaginary axis (complex roots) are moved along the unit circle towards the real axis with small angle $\varepsilon.$ Then the level $\rho$ is slightly decreased with $\eta.$ This approach is used for polynomials of degree $d\geq 4.$ The quadratic polynomial case is shortly discussed below and cubic polynomial case is presented as a first example.

To this end we have to find the maximum $\eta$ for a chosen $\varepsilon$ so that the spectrum gets separated in only two parts, one on the right hand side of the imaginary axis and one on the left hand side. Also we can find the maximum angle $\alpha$ (see Figure \ref{Fig1}) such that the spectrum will lay inside our lemniscate. 
\newpage
\begin{figure}[h]
\centering
\includegraphics[width=0.68\textwidth]{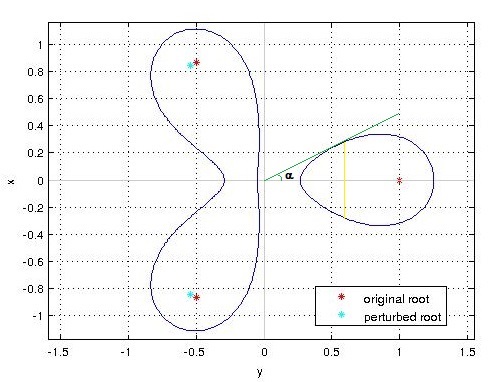}
\caption{Level 0.996}
\label{Fig1}
\end{figure}

 Now we analyze some cases. The quadratic polynomial, $p(z)=z^{2}-1$ is the classical lemniscate and for this one just has to decrease the magnitude $\rho$ to below 1. No perturbation of the roots is needed. For a decrease with $\eta=0.01 $ of the level we get the following picture:
\begin{figure}[h]
\centering
\includegraphics[width=0.55\textwidth]{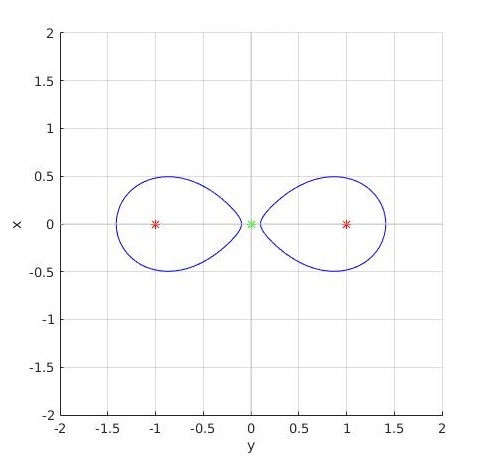}
\caption{Level 0.99}
\label{fig:deg3}
\end{figure}

\begin{example}
Let $p(z)=z^{3}-1$ with roots $e^{2\pi i/3},$ $e^{-2\pi i/3}$ and $1.$ We apply a perturbation with $\varepsilon$ to the complex roots and we get
$$ p_\varepsilon(z)=(z-e^{i(2\pi /3+\varepsilon)})(z-e^{-i(2\pi /3+\varepsilon}))(z-1) .$$
Thus the lemniscate is $|p_\varepsilon(z)|=1-\eta.$ For $\varepsilon=\pi/70$ the resulted picture is shown in Figure \ref{figure3}.
\newpage
\begin{figure}[hb]
    \centering
    \begin{subfigure}[b]{0.49\textwidth}
        \includegraphics[width=\textwidth]{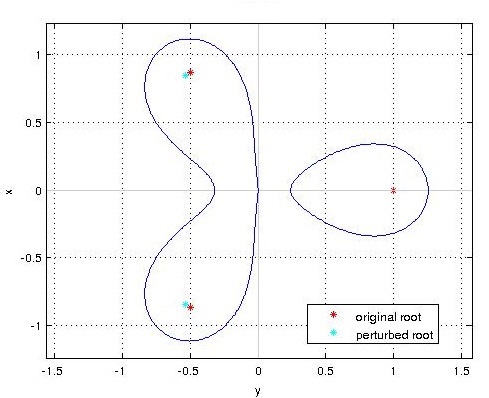}
        \caption{Level 1}
        \label{fig3a}
    \end{subfigure}
    \begin{subfigure}[b]{0.49\textwidth}
        \includegraphics[width=\textwidth]{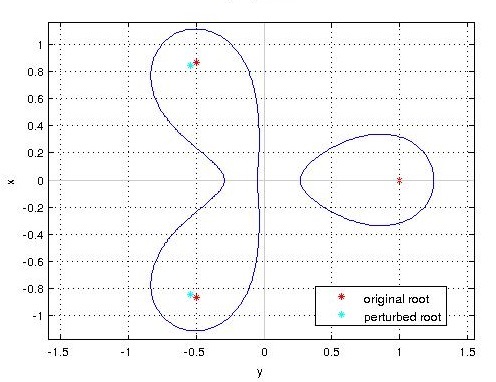}
        \caption{Level 0.996}
        \label{fig3b}
    \end{subfigure}
    \caption{Separation of the cubic polynomial}\label{figure3}
\end{figure}
\end{example} 


\begin{example}
Let $p(z)=z^{4}-1.$ Since there are roots on the imaginary axis, we have to apply a rotation with $\pi/4.$ Thus our polynomial becomes $p(z)=z^{4}+1.$ 

Then
$$ p(z)=(z-e^{i \theta})(z-e^{-i\theta})(z+e^{i\theta})(z+e^{-i\theta})$$
where $\theta= \pi/4,$ so we have
\begin{equation}
p(z)= (z^{2}-e^{2i\theta})(z^{2}-e^{-2i\theta})=z^{4}-2\cos (2\theta)z^{2}+1. \label{eq3.2}
\end{equation} 
Next, we decrease the angle $\theta$ with $\varepsilon$ small enough and denote the new angle $\theta_\varepsilon=\theta-\varepsilon.$ For the polynomial with perturbed $\theta$ we compute the lemniscate.

We have that 
\begin{eqnarray}
e^{2i\theta_\varepsilon} &=& e^{2i\theta-2i\varepsilon}=e^{i\pi/2-2i\varepsilon} \nonumber\\
&=& i(1-2i\varepsilon -2\varepsilon^{2}+\dots)=2\varepsilon +i(1-2\varepsilon^{2}+\dots). \nonumber
\end{eqnarray}
For $z=t\in \mathbb{R}$
\begin{eqnarray}
|p_{\theta_\varepsilon}(t)| &=& |t^{2}-2\varepsilon +i(-1+2\varepsilon^{2}+\dots)|^{2} \nonumber\\
&=& t^{4}-4t^{2}\varepsilon +1+ o(\varepsilon^{2}). \label{eq3.3}
\end{eqnarray}
For $z=it\in \mathbb{C}$
\begin{eqnarray}
|p_{\theta_\varepsilon}(it)| &=& |(it)^{2}-2\varepsilon +i(-1+2\varepsilon^{2}+\dots)|^{2} \nonumber\\
&=& t^{4}-4t^{2}\varepsilon +1+ o(\varepsilon^{2}). \nonumber
\end{eqnarray}
We write $t=z=x+iy$ in (\ref{eq3.3}) and compute the lemniscate $l: |p_{\varepsilon}(z)|=1-\eta,$ where $p_{\varepsilon}(z)=z^{4}-2\cos (2\theta_\varepsilon)z^{2}+1$ from (\ref{eq3.2}). Thus 
\begin{align*}
l : &\text{ } (x^{2}+y^{2})^{4}+2(x^{2}+y^{2})^{2}-16x^{2}y^{2}+16\varepsilon^{2}(x^{2}+y^{2})^{2} \nonumber\\
 & -8\varepsilon (x^{6}+x^{4}y^{2}-x^{2}y^{4}-y^{6}+x^{2}-y^{2})=0. \label{eq3.4}
\end{align*}

For $\varepsilon=\pi/70$ the results are shown in Figure \ref{figure4}:
\newpage
\begin{figure}[hb]
    \centering
    \begin{subfigure}[b]{0.49\textwidth}
        \includegraphics[width=\textwidth]{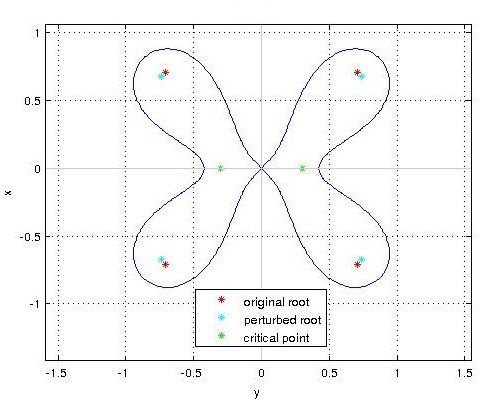}
        \caption{Level 1}
       \end{subfigure}
    \begin{subfigure}[b]{0.49\textwidth}
        \includegraphics[width=\textwidth]{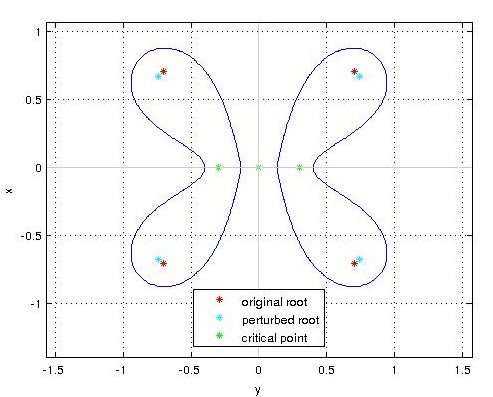}
        \caption{Level 0.997}
    \end{subfigure}
    \caption{Separation of the quartic polynomial}
    \label{figure4}
\end{figure}
\end{example}

Similar computations were made for polynomials of degree $6,8,10,12$ and $14,$ and these can be seen in Appendix \ref{AppendixD}.

\begin{remark} The goal was to find the maximum angle $\alpha$ such that the spectrum lays inside the lemniscate. For this, one can compute the ratio $a/b,$ where $a$ and $b$ are the lenght of the lines $a$ and $b$ from the picture below, and hence the angle $\alpha$ is 
\begin{equation}
\alpha=\arctan (a/b). \nonumber
\end{equation}
\begin{figure}[hb]
  \centering
    \includegraphics[width=0.6\textwidth]{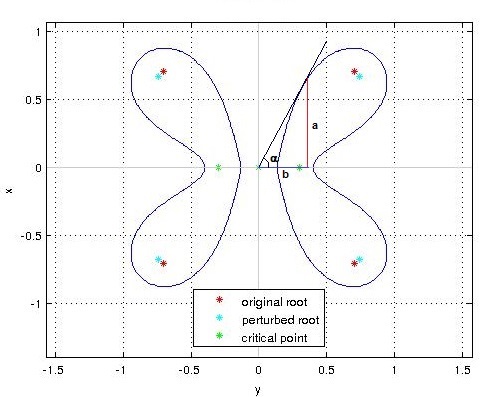}
  \caption{Ratio and angle $\alpha$}
  \label{Ratio}
\end{figure}

Note that, for a minimum level one might have that the line $a$ cuts the lemniscate, situation that happens even for a slightly perturbation of the level in all the cases with polynomials of degree $d\geq 6.$ Therefore, one has to consider a smaller angle.

For a minimum level $\rho$ of the lemniscate and a perturbation with $\varepsilon=\pi/70$ we have the following values for the ratio:
\newpage
\begin{table}[ht]
\caption{Ratio} \vspace{0.1 cm} 
\centering
\begin{tabular}{c| c c c c c c}
\hline\hline 
\text{ } & Degree 4 & Degree 6 & Degree 8 & Degree 10 & Degree 12 & Degree 14 \\ [0.3ex]
\hline
\textbf{a} & 0.5637 & 0.9040 & 0.9905 & 1.0043 & 0.9973 & 0.9846 \\
\textbf{b} & 0.3090 & 0.3767 & 0.3790 & 0.3624 & 0.3402 & 0.3176 \\
\textbf{a/b} & 1.8242 & 2.3997 & 2.6134 & 2.7712 & 2.9315 & 3.1001 \\
\hline
\end{tabular}
\label{table:ratio_MinLength}
\end{table}

\begin{figure}[h]
	\centering
		\includegraphics[width=0.70\textwidth]{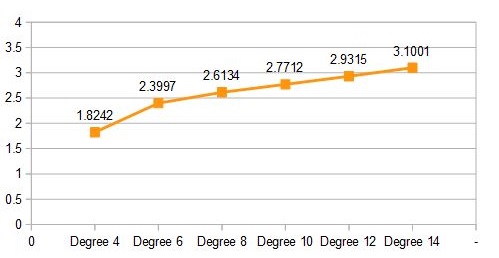}
	\caption{Ratio chart}
	\label{ratio_chart}
\end{figure}
\noindent The maximum angle $\alpha$ that we have found is presented in the following chart,
\begin{figure}[hb]
	\centering
		\includegraphics[width=0.70\textwidth]{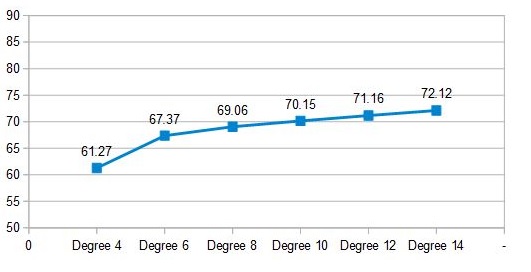}
	\caption{Maximum angle $\alpha$}
	\label{maxAngle}
\end{figure}
\end{remark}

\begin{remark} The quadratic polynomial is a special case since the only perturbation applied is decreasing the level. In this case, the maximum angle is when the level $\rho$ is unchanged, in this case $\alpha=45^{\circ}$ and the minimum angle goes to $0^{\circ}$ when the level is significantly decreased with $\eta=0,99.$ For a slightly change with $\eta=0.01$ we have found an angle of $29.92^{\circ}$ and for a change with $\eta=0.1$ we have registered an angle of $26.57^{\circ}.$
\end{remark}
\begin{remark}
For polynomials of degree $d\geq 4$, it is easy to check the maximum value of $\eta,$ i.e. the minimum value that the level can have, such that we get the desired separation. For example, if again the perturbation is $\varepsilon=\pi/70$ we get the following values: 
\newpage
\begin{table}[ht]
\centering 
\begin{tabular}{r| c c c c c c}
\hline\hline 
Degree  & 4 & 6 & 8 & 10 & 12 & 14 \\ 
\hline
$\eta$  & 0.008 & 0.0078 & 0.0038 & 0.0021 & 0.0022 & 0.0045 \\
\hline
\end{tabular}
\label{table:etaMax}
\end{table}

With these values the lemniscate squeezes next to the closest critical point to the perturbed root. A bigger decrease of the level would force that critical point to get out from the interior of the lemniscate and thus one doesn't get the desired separation. These estimates may not be the best but they are what we have reached by manipulating the pictures and the resulting pictures can be seen in Appendix \ref{AppendixE}.
\end{remark}

\section{Concluding remarks}

\hspace{0.7cm} In section $2$ we presented the closed formula and the bound for the Riesz projection, in Section $3$ we described the separation process and in this section one can see how the expansions on the Riesz projection look like when taking $\varphi=1$ for the components on the right side of the imaginary axis and $\varphi=-1$ for the ones on the left side. We will finish this paper with an example that shows what the effects are on the Riesz projection.

In the appendices of the technical report version of this paper one can find applications that explicitly compute the series expansions for $f_j$'s in the decomposition
$$
\varphi(z)= \sum_{j=1}^d \delta_j(z) f_j(w), \quad w=p(z),
$$
when $\varphi $ is identically $1$ on the right half plane and $-1$ on the left half plane, for the quadratic, the quartic, the perturbed quartic polynomial and the polynomial $q(z)=p(z)^{n}=(z^{2}-1)^{n}$, respectively.

\begin{remark}
In using multicentric calculus a central problem is to find a polynomial $p$ such that $p(A)$ has small norm and, when aiming for Riesz projection, that the lemniscate on the level the of $\| p(A)\| $ separates the spectrum into different components.
This can be done, for example, by minimizing $\| p(A) \|$ approximatively over a set of polynomials, or, by using a suitable $p$ which has been computed for a neigbouring matrix.

Alternatively, and that is the main topic here, one search for polynomials $p$ such that it is small in a neighbourhood of the spectrum of $A.$ And then computes heigh enough power $p(A)^{2^{m}}$ such that $\|p(A)^{2^{m}} \|^{1/2^{m}} \sim \rho (p(A)).$ 
\end{remark}

In the following example we point out with the help of a low-dimensional problem, how the size of coupling can affect on the need of taking a high power of $p(A).$

\begin{example}
\begin{equation} 
 A=\begin{pmatrix}
B &  X \\
0 &  -B   
\end{pmatrix}
\end{equation}
be a $4 \times 4$ matrix where 
\begin{equation}
B=\begin{pmatrix}
\alpha & 1 \\
0 & \alpha
\end{pmatrix}
\end{equation}
and
\begin{equation}
X=\begin{pmatrix}
0 & \gamma \\
\gamma & 0
\end{pmatrix}.
\end{equation}
In this example we could take $p(z)=z^2-\alpha^2$ to actually get a closed form  for the projection. However,   we take $p(z)=z^2-1$ as our polynomial and then the effect of $\alpha>0$ being close or further away from 1 models the lack of exact knowledge on the spectrum.   We are interested in having $\|p(A)^n\| <1$ and ask how the parameters $\alpha$ and $\gamma$  contribute to the value of $n$ needed.
Qualitatively it is clear that such an $n$ exists if and only if $ \alpha < \sqrt 2$, independently of $\gamma$.

Substituting $A$ into $p$ we have
$$
p(A)= \begin{pmatrix}
C & Y\\
0& C
\end{pmatrix},
$$
where
$$
C=\begin{pmatrix}
\alpha^2-1 & 2\alpha\\
0& \alpha^2 -1
\end{pmatrix}
$$
and 
$$
Y=\begin{pmatrix}
\gamma & 0 \\
0 & - \gamma
\end{pmatrix}.
$$
A short computation shows that
$$
p(A)^n = \begin{pmatrix}
C^n &    n(\alpha^2-1)^{n-1} Y\\
0& C^n
\end{pmatrix}.
$$
 Thus,  we have
 $$
 \|p(A)^n\|  \sim |\alpha^2-1|^{n-1} \Big[ |\alpha^2-1| + n(|\alpha| + |\gamma|) \Big],
 $$
so that if $| \alpha ^2 - 1| << 1$ then a small $n$ shall work. If however, $|\alpha^2-1|= 1- \varepsilon$ with $0<\varepsilon <<1$, modelling the case when e.g. spectrum  of $A$ is scattered inside the lemniscate, then the behavior is of the nature
$$
\|p(A)^n\| \sim (1-\varepsilon)^n(n+1),
$$
which becomes  below $1$ only for $n >> 1/\varepsilon$.
\end{example}

\newcommand{\bibname}{Bibliography}
\bibliographystyle{plain}

\addcontentsline{toc}{section}{Bibliography}

\newpage
\begin{appendices}
{\Huge \appendixname}
\section{Codes}
\subsection{Mathematica code used for Remark 2.10 }
\label{AppendixA}
In Remark 2.10 we have presented the values of the constants $C$ for polynomials of degrees $4,6,8,10,12$ and $14$ and these computations were made with Mathematica in the following way.

To compute $C$ one needs to compute $\sum |\delta_l(\lambda_k,w)|$ where 
$$\delta_l(\lambda,w)= \frac{p(\lambda)-w}{p'(\zeta_l(w))(\lambda-\zeta_l(w))} $$ and where $\zeta_l(w)$ are the roots of $p(\lambda)-w=0.$

We present the case when the polynomial $p$ is the perturbed quartic polynomial, i.e. $p(z)=z^{4}-2 z^{2} \sin(2\varepsilon)+1.$ Similar codes were developed for the perturbed monic polynomials of degrees $6,8,10,12$ and $14.$

First we compute the roots of our polynomial
{\small 
\begin{verbatim}
l = NSolve[(z - E^(I (Pi/4 - Pi/70)))(z - E^(-I (Pi/4 - Pi/70)))
  (z + E^(I (Pi/4 - Pi/70)))(z + E^(-I ( Pi/4 - Pi/70)))== 0, z] 

l1 = N[Part[l, 1, 1, 2]];     l2 = N[Part[l, 2, 1, 2]]
l3 = N[Part[l, 3, 1, 2]];     l4 = N[Part[l, 4, 1, 2]]
\end{verbatim} }
then we compute $\zeta_l(w)$
{\small 
\begin{verbatim}
xi = Solve[(z - E^(I (Pi/4 - Pi/70)))(z - E^(-I (Pi/4 - Pi/70)))
     (z + E^(I (Pi/4 - Pi/70)))(z + E^(-I ( Pi/4 - Pi/70))) == w,
      z] /. w -> 0.992

xi1 = Part[xi, 1, 1, 2]		    xi3 = Part[xi, 3, 1, 2]
xi2 = Part[xi, 2, 1, 2]		    xi4 = Part[xi, 4, 1, 2].
\end{verbatim}}
We also need the derivative of $p,$ which will be computed with the following line
{\small 
\begin{verbatim}
ped[z_] := D[(z - E^(I (Pi/4 - Pi/70))) (z - E^(-I (Pi/4 - Pi/70)))
        (z + E^(I (Pi/4 - Pi/70)))(z + E^(-I ( Pi/4 - Pi/70))), z];
\end{verbatim}}
\noindent and all these are needed to calculate each $ \delta_l(\lambda_k,w) .$ Here are just $\delta_l(\lambda_1,w),$ for $l=1,\dots ,4,$ since by replacing $\lambda_1$ with the others one finds all of them:
 {\small 
\begin{verbatim}
d1l1 = 0.992/((ped[x] /. x -> xi1) (xi1 - l1))
d2l1 = 0.992/((ped[x] /. x -> xi2) (xi2 - l1))	   
d3l1 = 0.992/((ped[x] /. x -> xi3) (xi3 - l1))	   
d4l1 = 0.992/((ped[x] /. x -> xi4) (xi4 - l1)).
\end{verbatim}}
After all $\delta_l(\lambda_k,w)$ are computed, one just has to sum up the absolute values of them,
 {\small 
\begin{verbatim}
Abs[d1l1] + Abs[d2l1] + Abs[d3l1] + Abs[d4l1] + Abs[d1l2] + 
Abs[d2l2] + Abs[d3l2] + Abs[d4l2] + Abs[d1l3] + Abs[d2l3] + 
Abs[d3l3] + Abs[d4l3] + Abs[d1l4] + Abs[d2l4] + Abs[d3l4] + 
Abs[d4l4]

2293.81
\end{verbatim}}
\noindent and finally the constant $C$ is this value multiplied with $s,$ the distance from the lemniscate to the critical point outside the lemniscate. This $s$ was computed with the program from the next subsection.

\subsection{Matlab program - computes $s$ from Remark 2.10} 
\label{AppendixB}
{\small
\begin{verbatim}
function pituus = minlength(piste,polyn,rho,loota)
  %UNTITLED piste=[x;y] (column vector) is the critical point
  % and polyn=[a_n, .. , a_1,a_0] (row vector) for the REAL   
  % polynomial an*z^n + ... + a1*z+a0, rho is the level 
  % (p(z)=rho gives the lemniscate)
  % loota=the size of the drawing area. MAKE BIG ENOUGH FOR THE
  % LEMNISCATE TO FIT IN, oterwise the Matlab's contour command 
  % returns only partof the lemniscate drawn and you get wrong 
  % (and strange) answers
  
% get the level curve data 
C = kontour(polyn,[rho],[-loota,-loota,loota,loota]);
[~,m]=size(C);
% (copypaste from lemnlength:)
katkot=[];
for ii=2:m
  if (C(2,ii)>10)
      katkot=[katkot, ii];
  end  
end
% build M matrix that contains the start and end points
% for the pieces of the lemniscate in the data
valienlkm=length(katkot)+1;
M=zeros(2,valienlkm); M(1,1)=2; M(2,valienlkm)=m;
for ii=1:(valienlkm-1)
   M(2,ii)=katkot(ii)-1; M(1,ii+1)=katkot(ii)+1; 
end
% go through M and calculate all the angles
pituudet=sqrt(2)*loota*ones(1,m); 
for ii=1:valienlkm
   for jj=M(1,ii):(M(2,ii))
      pituudet(jj)=norm([C(1,jj)-piste(1,1),C(2,jj)-piste(2,1)]);
   end
end
pituus=min(pituudet);
end
\end{verbatim}
}
The programme above uses the following programme:
{\small
\begin{verbatim}
function C = kontour(p,tasot,ruutu)
 % UNTITLED Draws the lemniscates of the polynomial p. 
 % The picture coordinates are in ruutu vector.
 % ruutu=[xmin,ymin,xmax,ymax]
 % tasot=[l1,l2,..,ln] (the levels to be drawn)
 % p=[an,...,a1,a0] (the polynomial an*z^n + ... + a1*z+a0)

juuret=roots(p);
xx=linspace(ruutu(1),ruutu(3),501);
yy=linspace(ruutu(2),ruutu(4),501);
[X,Y]=meshgrid(xx,yy);
Z=abs(polyval(p,X+1i*Y));
if length(tasot)==1
  [C,~]=contour(X,Y,Z,[tasot(1),tasot(1)]);
else
  [C,~]=contour(X,Y,Z,tasot);
end
hold on, plot(real(juuret),imag(juuret),'.k')
end
\end{verbatim}
}

\section{Expansions for the Riesz projection} \label{AppendixC}
\subsection{The quadratic polynomial } \label{AppendixC1}

\noindent Let $p(z)=z^{2}-1$ with solutions $\lambda_1= 1,$ and $\lambda_2=-1.$ Denoting $\delta_1(z)=(1+z)/2$ and $\delta_2(z) = (1-z)/2$ we obtain
\begin{eqnarray}
f_1(z^{2}-1)&=& \frac{1}{2} [f(z)+f(-z)]+\frac{f(z)-f(-z)}{2z} \label{f1d2} \\
f_2(z^{2}-1)&=& \frac{1}{2} [f(z)+f(-z)]-\frac{f(z)-f(-z)}{2z} \label{f2d2}.
\end{eqnarray}

\noindent Consider the Riesz projection which is obtained by assuming $\varphi$ to be identically $1$ near $1$ and $-1$ near $-1.$ We have, for $|w|<1,$ where $w=z^{2}-1$ 
$$
(w+1)^{1/2}= 1+\frac{1}{2}w-\frac{1}{8}w^{2}+\frac{1}{16}w^{3}+\cdots 
$$
\noindent and
$$
 (w+1)^{-1/2}= 1-\frac{1}{2}w+\frac{3}{8}w^{2}-\frac{5}{16}w^{3}+\cdots .
 $$

\noindent Let us compute the two-centric representation first around the point $1.$ There with $z=(w+1)^{1/2}$
\begin{eqnarray}
\delta_1(z)&=& 1+\frac{1}{4}w -\frac{1}{16}w^{2} +\frac{1}{32}w^{3}+\cdots , \nonumber\\
\delta_2(z)&=& -\frac{1}{4}w +\frac{1}{16}w^{2} -\frac{1}{32}w^{3}+\cdots .\nonumber
\end{eqnarray}

\noindent From (\ref{f1d2}) and (\ref{f2d2}) we obtain
\begin{eqnarray}
f_1(w) &=& 1-\frac{1}{2}w +\frac{3}{8}w^{2}-\frac{5}{16}w^{3} +\cdots , \nonumber\\
f_2(w) &=& -1+\frac{1}{2}w -\frac{3}{8}w^{2}+\frac{5}{16}w^{3} +\cdots .\nonumber
\end{eqnarray}

\noindent This gives

\begin{eqnarray}
\delta_1(z)f_1(w) &=& 1-\frac{1}{4}w +\frac{3}{16}w^{2} -\frac{5}{32}w^{3}+\cdots ,\nonumber\\
\delta_2(z)f_2(w) &=& \frac{1}{4}w -\frac{3}{16}w^{2} +\frac{5}{32}w^{3} +\cdots \nonumber
\end{eqnarray}

\noindent so their sum is identically $1.$ Near $-1$ we have 
\begin{eqnarray}
\delta_1(z)&=& -\frac{1}{4}w +\frac{1}{16}w^{2} -\frac{1}{32}w^{3}+\cdots ,\nonumber\\
\delta_2(z)&=& 1+\frac{1}{4}w -\frac{1}{16}w^{2} +\frac{1}{32}w^{3}+\cdots , \nonumber
\end{eqnarray}

\noindent which gives
\begin{eqnarray}
\delta_1(z)f_1(w) &=& -\frac{1}{4}w +\frac{3}{16}w^{2} -\frac{5}{32}w^{3}+\cdots ,\nonumber\\
\delta_2(z)f_2(w) &=& -1+\frac{1}{4}w -\frac{3}{16}w^{2} +\frac{5}{32}w^{3} +\cdots .\nonumber
\end{eqnarray}

\noindent So, near $-1$ their sum is identically $-1.$

\subsection{The quartic polynomial }\label{AppendixC2}

\noindent Let $p(z)=z^{4}+1$ with roots
$$\lambda_1=(1+i)/ \sqrt{2},\text{ } \lambda_2=(-1+i)/ \sqrt{2},\text{ } \lambda_3=(-1-i)/ \sqrt{2},\text{ }\lambda_4=(1-i)/ \sqrt{2}.$$ Denoting 
\begin{eqnarray}
\delta_1 (z)&=& \tfrac{(-1-i)z^{3}-i\sqrt{2} z^{2}+(1-i)z+\sqrt{2}}{4\sqrt{2}}\nonumber\\
&=& (-\tfrac{1}{8}-\tfrac{i}{8})((1+i)+\sqrt{2}z)(z^{2}+i) \nonumber\\
\delta_2 (z)&=& \tfrac{(1-i)z^{3}+i\sqrt{2} z^{2}-(1+i)z+\sqrt{2}}{4\sqrt{2}}
\nonumber\\
&=& (\tfrac{1}{8}+\tfrac{i}{8})((1+i)-\sqrt{2}z)(z^{2}-i) \nonumber\\
\delta_3 (z)&=& \tfrac{(1+i)z^{3}-i\sqrt{2} z^{2}-(1-i)z+\sqrt{2}}{4\sqrt{2}}\nonumber\\
&=& (-\tfrac{1}{8}-\tfrac{i}{8})((1+i)-\sqrt{2}z)(z^{2}+i) \nonumber\\
\delta_4 (z)&=& \tfrac{(-1+i)z^{3}+i\sqrt{2} z^{2}+(1+i)z+\sqrt{2}}{4\sqrt{2}}\nonumber\\
&=& (\tfrac{1}{8}+\tfrac{i}{8})((1+i)+\sqrt{2}z)(z^{2}-i) \nonumber
\end{eqnarray}

\noindent we obtain

\begin{eqnarray} 
f_1(z^{4}+1) &=& \tfrac{1}{4z^{3}} [ (\tfrac{-1+i}{\sqrt{2}}+iz+\tfrac{1+i}{\sqrt{2}} z^{2}+z^{3})f(z)+ (\tfrac{-1-i}{\sqrt{2}}-iz+\tfrac{1-i}{\sqrt{2}} z^{2}+z^{3})f(iz) \nonumber\\
&\text{ }&+  (\tfrac{1-i}{\sqrt{2}}+iz-\tfrac{1+i}{\sqrt{2}} z^{2}+z^{3})f(-z)+ (\tfrac{1+i}{\sqrt{2}}-iz-\tfrac{1-i}{\sqrt{2}} z^{2}+z^{3})f(-iz)] \nonumber\\
\label{f1}\\
f_2(z^{4}+1)&=& \tfrac{1}{4z^{3}} [ (\tfrac{1+i}{\sqrt{2}}-iz-\tfrac{1-i}{\sqrt{2}} z^{2}+z^{3})f(z)+ (\tfrac{-1+i}{\sqrt{2}}+iz+\tfrac{1+i}{\sqrt{2}} z^{2}+z^{3})f(iz) \nonumber\\
&\text{ }&+  (\tfrac{-1-i}{\sqrt{2}}-iz+\tfrac{1-i}{\sqrt{2}} z^{2}+z^{3})f(-z)+ (\tfrac{1-i}{\sqrt{2}}+iz-\tfrac{1+i}{\sqrt{2}} z^{2}+z^{3})f(-iz)]\nonumber\\ \label{f2}
\end{eqnarray}
\begin{eqnarray}
f_3(z^{4}+1) &=& \tfrac{1}{4z^{3}} [ (\tfrac{1-i}{\sqrt{2}}+iz-\tfrac{1+i}{\sqrt{2}} z^{2}+z^{3})f(z)+ (\tfrac{1+i}{\sqrt{2}}-iz-\tfrac{1-i}{\sqrt{2}} z^{2}+z^{3})f(iz) \nonumber\\
&\text{ }&+  (\tfrac{-1+i}{\sqrt{2}}+iz+\tfrac{1+i}{\sqrt{2}} z^{2}+z^{3})f(-z)+ (\tfrac{-1-i}{\sqrt{2}}-iz+\tfrac{1-i}{\sqrt{2}} z^{2}+z^{3})f(-iz)]\nonumber\\ \label{f3}\\
f_4(z^{4}+1) &=& \tfrac{1}{4z^{3}} [ (\tfrac{-1-i}{\sqrt{2}}-iz+\tfrac{1-i}{\sqrt{2}} z^{2}+z^{3})f(z)+ (\tfrac{1-i}{\sqrt{2}}+iz-\tfrac{1+i}{\sqrt{2}} z^{2}+z^{3})f(iz) \nonumber\\
&\text{ }&+  (\tfrac{1+i}{\sqrt{2}}-iz-\tfrac{1-i}{\sqrt{2}} z^{2}+z^{3})f(-z)+ (\tfrac{-1+i}{\sqrt{2}}+iz+\tfrac{1+i}{\sqrt{2}} z^{2}+z^{3})f(-iz)]\nonumber\\ \label{f4}
\end{eqnarray}

\noindent Consider the Riesz spectral projection which is obtained by assuming $\varphi$ to be identically $1$ near $\lambda_1$ and $\lambda_4,$ and $-1$ near $\lambda_2$ and $\lambda_3.$ We have, for $|w|<1,$ where $w=z^{4}+1$ 
$$ 
(w-1)^{1/4} = \frac{1+i}{\sqrt{2}} -\frac{1}{4} \frac{1+i}{\sqrt{2}} w -\frac{3}{32}\frac{1+i}{\sqrt{2}}w^{2}-\frac{7}{128}\frac{1+i}{\sqrt{2}}w^{3} +\cdots .
$$

\noindent Let us compute the four-centric representation first around $\lambda_1.$ There with $z=(w-1)^{1/4}$ we have
\begin{eqnarray}
\delta_1 &=& 1-\tfrac{3}{8}w-\tfrac{5}{64}w^{2}-\tfrac{5}{128}w^{3}+\cdots \nonumber\\
\delta_2 &=& (\tfrac{1}{8}-\tfrac{i}{8})w+\tfrac{1}{32}w^{2}+(\tfrac{1}{64}+\tfrac{i}{256})w^{3}+\cdots \nonumber\\
\delta_3 &=& \tfrac{1}{8}w+\tfrac{1}{64}w^{2}+\tfrac{1}{128}w^{3}+\cdots\nonumber\\
\delta_4 &=& (\tfrac{1}{8}+\tfrac{i}{8})w+\tfrac{1}{32}w^{2}+(\tfrac{1}{64}-\tfrac{i}{256})w^{3}+\cdots .\nonumber
\end{eqnarray}

\noindent From (\ref{f1}), (\ref{f2}), (\ref{f3}) and (\ref{f4}) we obtain
\begin{eqnarray}
f_1(w) &=& 1+(\tfrac{1}{2}-\tfrac{i}{4})w+(\tfrac{13}{32}-\tfrac{i}{4})w^{2}+(\tfrac{23}{64}-\tfrac{31 i}{128})w^{3}+\cdots , \nonumber\\
f_2(w) &=& -1+(-\tfrac{1}{2}-\tfrac{i}{4})w+(-\tfrac{13}{32}-\tfrac{i}{4})w^{2}+(-\tfrac{23}{64}-\tfrac{31 i}{128})w^{3}+\cdots ,\nonumber\\
f_3(w) &=& -1+(-\tfrac{1}{2}+\tfrac{i}{4})w+(-\tfrac{13}{32}+\tfrac{i}{4})w^{2}+(-\tfrac{23}{64}+\tfrac{31 i}{128})w^{3}+\cdots ,\nonumber\\
f_4(w) &=& 1+(\tfrac{1}{2}+\tfrac{i}{4})w+(\tfrac{13}{32}+\tfrac{i}{4})w^{2}+(\tfrac{23}{64}+\tfrac{31 i}{128})w^{3}+\cdots .\nonumber
\end{eqnarray}

\noindent This gives
\begin{eqnarray}
\delta_1(z)f_1(w) &=& 1+(\tfrac{1}{8}-\tfrac{i}{4})w+(\tfrac{9}{64}-\tfrac{5i}{32})w^{2}+(\tfrac{33}{256}-\tfrac{33 i}{256})w^{3}+\cdots , \nonumber\\
\delta_2(z)f_2(w) &=& -(\tfrac{1}{8}-\tfrac{i}{8})w-(\tfrac{1}{8}-\tfrac{i}{32})w^{2}-(\tfrac{29}{256}-\tfrac{i}{128})w^{3}+\cdots ,\nonumber\\
\delta_3(z)f_3(w) &=& -\tfrac{1}{8}w-(\tfrac{5}{64}-\tfrac{i}{32})w^{2}-(\tfrac{17}{256}-\tfrac{9 i}{256})w^{3}+\cdots ,\nonumber\\
\delta_4(z)f_4(w) &=& (\tfrac{1}{8}+\tfrac{i}{8})w+(\tfrac{1}{16}+\tfrac{3i}{32})w^{2}+(\tfrac{13}{256}+\tfrac{11 i}{128})w^{3}+\cdots \nonumber
\end{eqnarray}
\noindent so their sum is identically $1.$ Near $\lambda_2$ we have

$$ 
\delta_1= \delta_4(z), \quad \delta_2=\delta_1(z) \quad \delta_3=\delta_2(z) \quad \delta_4=\delta_3(z)
$$
\noindent which gives
\begin{eqnarray}
\delta_1(z)f_1(w) &=& (\tfrac{1}{8}+\tfrac{i}{8})w+(\tfrac{1}{8}+\tfrac{i}{32})w^{2}+(\tfrac{29}{256}+\tfrac{i}{128})w^{3}+\cdots , \nonumber\\
\delta_2(z)f_2(w) &=& -1-(\tfrac{1}{8}+\tfrac{i}{4})w-(\tfrac{9}{64}+\tfrac{5i}{32})w^{2}-(\tfrac{33}{256}+\tfrac{33i}{256})w^{3}+\cdots ,\nonumber\\
\delta_3(z)f_3(w) &=& -(\tfrac{1}{8}-\tfrac{i}{8})w-(\tfrac{1}{16}-\tfrac{3i}{32})w^{2}-(\tfrac{13}{256}-\tfrac{11 i}{128})w^{3}+\cdots ,\nonumber\\
\delta_4(z)f_4(w) &=& \tfrac{1}{8}w+(\tfrac{5}{64}+\tfrac{i}{32})w^{2}+(\tfrac{17}{256}+\tfrac{9 i}{256})w^{3}+\cdots .\nonumber
\end{eqnarray}
\noindent So, near $\lambda_2$ their sum is $-1.$ Near $\lambda_3$ we have
$$ 
\delta_1= \delta_3(z), \quad \delta_2=\delta_4(z) \quad \delta_3=\delta_1(z) \quad \delta_4=\delta_2(z)
$$
\noindent which gives
\begin{eqnarray}
\delta_1(z)f_1(w) &=& \tfrac{1}{8}w+(\tfrac{5}{64}-\tfrac{i}{32})w^{2}+(\tfrac{17}{256}-\tfrac{9i}{256})w^{3}+\cdots , \nonumber\\
\delta_2(z)f_2(w) &=& -(\tfrac{1}{8}+\tfrac{i}{8})w-(\tfrac{1}{16}+\tfrac{3i}{32})w^{2}-(\tfrac{13}{256}+\tfrac{11i}{128})w^{3}+\cdots ,\nonumber\\
\delta_3(z)f_3(w) &=& -1-(\tfrac{1}{8}-\tfrac{i}{4})w-(\tfrac{9}{64}-\tfrac{5i}{32})w^{2}-(\tfrac{33}{256}-\tfrac{33 i}{256})w^{3}+\cdots ,\nonumber\\
\delta_4(z)f_4(w) &=& (\tfrac{1}{8}-\tfrac{i}{8})w+(\tfrac{1}{8}-\tfrac{i}{32})w^{2}+(\tfrac{29}{256}-\tfrac{i}{128})w^{3}+\cdots \nonumber
\end{eqnarray}
so their sum is identically $-1$ Near $\lambda_4$ we have
$$
 \delta_1= \delta_2(z), \quad \delta_2=\delta_3(z) \quad \delta_3=\delta_4(z) \quad \delta_4=\delta_1(z)
$$
\noindent which gives
\begin{eqnarray}
\delta_1(z)f_1(w) &=& (\tfrac{1}{8}-\tfrac{i}{8})w+(\tfrac{1}{16}-\tfrac{3i}{32})w^{2}+(\tfrac{13}{256}-\tfrac{11i}{128})w^{3}+\cdots , \nonumber\\
\delta_2(z)f_2(w) &=& -\tfrac{1}{8}w-(\tfrac{5}{16}+\tfrac{i}{32})w^{2}-(\tfrac{17}{256}+\tfrac{9i}{256})w^{3}+\cdots ,\nonumber\\
\delta_3(z)f_3(w) &=& -(\tfrac{1}{8}+\tfrac{i}{8})w-(\tfrac{1}{8}+\tfrac{i}{32})w^{2}-(\tfrac{29}{256}+\tfrac{i}{128})w^{3}+\cdots ,\nonumber\\
\delta_4(z)f_4(w) &=& 1+(\tfrac{1}{8}+\tfrac{i}{4})w+(\tfrac{9}{64}+\tfrac{5i}{32})w^{2}+(\tfrac{33}{256}+\tfrac{33i}{256})w^{3}+\cdots .\nonumber
\end{eqnarray}
So, near $\lambda_4$ their sum is $1.$

\subsection{The perturbed quartic polynomial }\label{AppendixC3}

Now let's perturb the roots of $p(z)=z^{4}+1$ with $\varepsilon.$ Therefore, our polynomial becomes $$ p_\varepsilon(z)=z^{4}-2 z^{2} \sin(2\varepsilon)+1 $$ 
with roots 
$$
\begin{array}{llll}
\lambda_1= e^{i(\pi /4-\varepsilon)}, & \lambda_2=-e^{-i(\pi /4-\varepsilon)}, &
\lambda_3=-e^{i(\pi /4-\varepsilon)} , & \lambda_4=e^{-i(\pi /4-\varepsilon)} ,
\end{array}
$$
and derivative $p_\varepsilon'(z)=4 z^{3}-4 z \sin(2\varepsilon).$ 

Denoting
\begin{eqnarray}
\delta_1(z) &=& (\tfrac{1}{8}-\tfrac{i}{8})((1+i)+\sqrt{2}e^{i\varepsilon}z)(e^{2i\varepsilon}-iz^{2})\sec(2\varepsilon) \nonumber\\
\delta_2(z) &=& \tfrac{e^{-3i\varepsilon}}{4\sqrt{2}}(\sqrt{2}e^{i\varepsilon}-(1+i)z)(1+ie^{2i\varepsilon}z^{2})\sec(2\varepsilon)\nonumber\\
\delta_3(z) &=& (\tfrac{1}{8}+\tfrac{i}{8})((-1-i)+\sqrt{2}e^{i\varepsilon}z)(ie^{2i\varepsilon}+z^{2})\sec(2\varepsilon)\nonumber\\
\delta_4(z) &=& \tfrac{e^{-3i\varepsilon}}{4\sqrt{2}}(\sqrt{2}e^{i\varepsilon}+(1+i)z)(1+ie^{2i\varepsilon}z^{2})\sec(2\varepsilon)\nonumber 
\end{eqnarray}
\noindent we get
\begin{eqnarray}
f_1(p_\varepsilon(z)) &=& \tfrac{i}{\sqrt{2}z^{3}} e^{-3i\varepsilon}(1+e^{2i\varepsilon}z^{2})\label{f1d4eps}\\
f_2(p_\varepsilon(z)) &=& \tfrac{1}{\sqrt{2}z^{3}} e^{i\varepsilon}(e^{2i\varepsilon}-z^{2})\label{f2d4eps}\\
f_3(p_\varepsilon(z)) &=& \tfrac{-i}{\sqrt{2}z^{3}} e^{-3i\varepsilon}(1+e^{2i\varepsilon}z^{2})\label{f3d4eps}\\
f_4(p_\varepsilon(z)) &=& \tfrac{-1}{\sqrt{2}z^{3}} e^{i\varepsilon}(e^{2i\varepsilon}-z^{2}).\label{f4d4eps}
\end{eqnarray}

\noindent Considering the Riesz projection which is obtained by assuming $\varphi$ to be identically $1$ near the roots on the right hand side of the imaginary axis and $-1$ near the others, we have, for $|w|<1,$ where $w=p_\varepsilon (z),$

\begin{eqnarray}
z &=&  (-1)^{1/4}\left(1-i\varepsilon-\tfrac{1}{2}\varepsilon^{2}+\tfrac{i}{6}\varepsilon^{3}+\dots \right)\nonumber\\
&\text{ }& -\tfrac{1}{4}(-1)^{1/4} \left( 1+i\varepsilon+\tfrac{3}{2}\varepsilon^{2}+\tfrac{11i}{6}\varepsilon^{3}+\dots \right)w \nonumber\\
&\text{ }& +\tfrac{1}{32}(-1)^{3/4} \left( 1+3i\varepsilon-\tfrac{1}{2}\varepsilon^{2}+\tfrac{15i}{5}\varepsilon^{3} +\dots \right)w^{2} +\dots \nonumber
\end{eqnarray}

\noindent Let us compute the four-centric representation first around $\lambda_1.$ There with $w=z^{4}-2 z^{2} \sin(2\varepsilon)+1$
\begin{eqnarray}
\delta_1(z) &=& 1-\tfrac{1}{8}\left(3+2i\varepsilon +8\varepsilon^{2}+\tfrac{8i}{3}\varepsilon^{3}+\dots \right)w+\tfrac{1}{32}\left( \left(2+\tfrac{3i}{2}\right)-(4-4i)\varepsilon \right. \nonumber\\
&\text{ }&+ \left. \left(8+2i\right)\varepsilon^{2}-\left(\tfrac{40}{3}-\tfrac{64i}{3}\right)\varepsilon^{3}+\dots \right)w^{2}+\dots \nonumber\\
\delta_2(z) &=& \tfrac{1}{8}\left( \left(1-i\right)-2\varepsilon +\left(4-2i\right) \varepsilon^{2}-\tfrac{20}{3}\varepsilon^{3}+\dots \right)w \nonumber\\
&\text{ }&- \tfrac{1}{32}\left( (1-i)-(1-i)\varepsilon +(5-7i)\varepsilon^{2}-\left( \tfrac{16}{3}-\tfrac{16i}{3}\right) \varepsilon^{3} +\dots \right)w^{2}+\dots \nonumber\\
\delta_3(z) &=& \tfrac{1}{8}\left(1+2i\varepsilon+\tfrac{8i}{3}\varepsilon^{3}+\dots \right)w- \tfrac{1}{32}\left( \left( 1+ \tfrac{i}{2}\right)-\left(2-2i\right)\varepsilon  \right. \nonumber\\
&\text{ }&+ \left. \left(4-2i\right)\varepsilon^{2}-\left(\tfrac{8}{3} - \tfrac{32i}{3} \right) \varepsilon^{3} +\dots\right) w^{2}+\dots \nonumber\\
\delta_4(z) &=& \tfrac{1}{8}\left( \left(1-i\right)+2\varepsilon +\left(4+2i\right) \varepsilon^{2}+\tfrac{20}{3}\varepsilon^{3}+\dots \right)w \nonumber\\
&\text{ }&- \tfrac{1}{32}\left(2i-(1-i)\varepsilon -(1-11i)\varepsilon^{2}-\left( \tfrac{16}{3}-\tfrac{16i}{3}\right)\varepsilon^{3}+\dots \right)w^{2}+\dots .\nonumber
\end{eqnarray}

\noindent From (\ref{f1d4eps}), (\ref{f2d4eps}), (\ref{f3d4eps}) and (\ref{f4d4eps}) we obtain
\begin{eqnarray}
f_1(w)&=& 1+\left( \left(\tfrac{1}{2}-\tfrac{i}{4}\right)+\left(\tfrac{1}{2}+i\right)\varepsilon+ \dots \right)w \nonumber\\
&\text{ }&+\left( \left(\tfrac{3}{16}-\tfrac{7i}{32}\right)+\left(\tfrac{7}{8}-\tfrac{3i}{4}\right)\varepsilon- \left(\tfrac{3}{4}-\tfrac{7i}{8}\right) \varepsilon^2 +\dots \right)w^{2}+\dots \nonumber\\
f_2(w)&=& \left( -1+(2-4i)\varepsilon+(10+8i)\varepsilon^2-\dots \right) \nonumber\\
&\text{ }&- \left( \left(\tfrac{1}{2}+\tfrac{i}{4}\right)-\left(\tfrac{5}{2}-\tfrac{7i}{2}\right)\varepsilon - \left(12-\tfrac{21i}{2}\right) \varepsilon^2 +\dots \right) w \nonumber\\ 
&\text{ }&- \left( \left(\tfrac{1}{4}+\tfrac{3i}{32}\right)-\left(\tfrac{9}{8}-\tfrac{39i}{16}\right)\varepsilon -\left(11+\tfrac{93i}{16}\right) \varepsilon^2 +\dots \right)w^{2}+\dots \nonumber\\ 
f_3(w)&=& -1-\left( \left(\tfrac{1}{2}-\tfrac{i}{4}\right)+\left(\tfrac{1}{2}-i\right)\varepsilon + \dots \right)w \nonumber\\
&\text{ }&-\left( \left(\tfrac{3}{16}-\tfrac{7i}{32}\right)+\left(\tfrac{7}{8}+\tfrac{3i}{4}\right)\varepsilon-\left(\tfrac{3}{4}-\tfrac{7i}{8}\right) \varepsilon^2 +\dots \right)w^{2}+\dots \nonumber\\
f_4(w)&=&\left(1-(2-4i)\varepsilon -(10+8i)\varepsilon^2+ \dots \right) \nonumber\\
&\text{ }&+ \left( \left(\tfrac{1}{2}+\tfrac{i}{4}\right)-\left(\tfrac{5}{2}-\tfrac{7i}{2}\right)\varepsilon - \left(12+\tfrac{21i}{2}\right) \varepsilon^2 +
\dots \right)w \nonumber\\ 
&\textbf{ }&+ \left( \left(\tfrac{1}{4}+\tfrac{3i}{32}\right)-\left(\tfrac{9}{8}-\tfrac{39i}{16}\right)\varepsilon -\left(11+\tfrac{93i}{16}\right) \varepsilon^2 +
\dots \right)w^{2}+\dots \nonumber
\end{eqnarray}

\noindent This gives
\begin{eqnarray}
\delta_1(z)f_1(w) &=& 1+\tfrac{w}{8}\left((1-2i)+(4+6i)\varepsilon -8\varepsilon^2  +\dots \right)\nonumber\\
&\text{ }&+\tfrac{w^2}{32}\left( \left(2-\tfrac{5i}{2}\right)+(16+12i)\varepsilon -(24-34i)\varepsilon^2  +\dots \right)+\dots \nonumber\\
\delta_2(z)f_2(w) &=& \tfrac{w}{8}\left( -(1-i)-6i\varepsilon +(10+8i)\varepsilon^2 +\dots \right)\nonumber\\
&\text{ }&-\tfrac{w^2}{32} \left( 2-(1-15i)\varepsilon -(45+11i)\varepsilon^2 +\dots \right) +\dots \nonumber\\
\delta_3(z)f_3(w) &=& \tfrac{w}{8}\left( -1-2i\varepsilon +\dots \right)\nonumber\\ 
&\text{ }&-\tfrac{w^2}{32} \left( \left(1-\tfrac{3i}{2}\right)+(6+6i)\varepsilon -(12-12i)\varepsilon^2 +\dots \right) + \dots \nonumber\\
\delta_4(z)f_4(w) &=& \tfrac{w}{8}\left( (1+i)-(4-2i)\varepsilon -(2+8i)\varepsilon^2  +\dots \right)\nonumber\\
&\text{ }&+\tfrac{w^2}{32}\left( (1+i)-(11-9i)\varepsilon -(33+39i)\varepsilon^2 +\dots \right) + \dots \nonumber
\end{eqnarray}
so their sum is identically $1.$ Near $\lambda_2$ we have
\begin{eqnarray}
\delta_1(z) &=& \left(-(1-i)\varepsilon -\left(\tfrac{4}{3}-\tfrac{4i}{3}\right)\varepsilon^3+\dots \right) \nonumber\\
&\text{ }& + \tfrac{w}{8}\left( (1+i)+2\varepsilon +(4+8i)\varepsilon^2 +\tfrac{8}{3}\varepsilon^3+\dots \right) \nonumber\\
&\text{ }& -\tfrac{w^2}{32} \left( 2-(4-i)\varepsilon -(4-8i)\varepsilon^2-\left(\tfrac{64}{3}+\tfrac{8i}{3}\right)\varepsilon^3+\dots \right) +\dots \nonumber\\
\delta_2(z) &=& \left( 1-3i\varepsilon -3\varepsilon^2+\dots \right)+ \tfrac{w}{8} \left( -3+4i\varepsilon -8\varepsilon^2+\tfrac{40i}{3}\varepsilon^3+\dots \right)  \nonumber\\
&\text{ }& +\tfrac{w^2}{32}\left( \left(2+\tfrac{3i}{2}\right)-(1-i)\varepsilon +(11+8i)\varepsilon^2- \left(\tfrac{16}{3}- \tfrac{16i}{3}\right) \varepsilon^3+\dots \right)+\dots \nonumber\\
\delta_3(z) &=& \left((1+i)\varepsilon +\left(\tfrac{4}{3}+\tfrac{4i}{3}\right)\varepsilon^3+\dots \right)\nonumber\\
&\text{ }& +\tfrac{w}{8} \left( (1-i)-2\varepsilon+(4-8i)\varepsilon^2-\tfrac{8}{3}\varepsilon^3\dots \right)\nonumber\\
&\text{ }& -\tfrac{w^2}{32} \left( (1-i)+(2+i)\varepsilon +(8-4i)\varepsilon^2+ \left(\tfrac{32}{3}+\tfrac{40i}{3}\right) \varepsilon^3 +\dots \right) +\dots \nonumber\\
\delta_4(z) &=& \left( i\varepsilon +3\varepsilon^2 -\tfrac{8i}{3}\varepsilon^3+\dots \right) +\tfrac{w}{8} \left( 1-4i\varepsilon -\tfrac{40i}{3}\varepsilon^3+\dots \right) \nonumber\\
&\text{ }& -\tfrac{w^2}{32} \left( \left(1+\tfrac{i}{2}\right)+(1-i)\varepsilon +(7+4i)\varepsilon^2 +\left(\tfrac{16}{3} -\tfrac{16i}{3}\right) \varepsilon^3+\dots \right)+...\nonumber
\end{eqnarray}
which gives
\begin{eqnarray}
\delta_1(z)f_1(w) &=& \left( -(1-i)\varepsilon -\left(\tfrac{4}{3}-\tfrac{4i}{3}\right)\varepsilon^3+\dots \right) \nonumber\\
&\text{ }& +\tfrac{w}{8} \left( (1+i)+6i\varepsilon -(8-4i)\varepsilon^2  +\dots \right) \nonumber\\
&\text{ }& +\tfrac{w^2}{32} \left( (3-i)+(7+16i)\varepsilon -(28-16i)\varepsilon^2 +\dots \right) + \dots \nonumber\\
\delta_2(z)f_2(w) &=& \left(-1+(2-i)\varepsilon +(1+2i)\varepsilon^2  +\dots \right)  \nonumber\\
&\text{ }& -\tfrac{w}{8} \left( (1+2i)-(8-8i)\varepsilon -(18+14i)\varepsilon^2  +\dots \right)\nonumber\\
&\text{ }& -\tfrac{w^2}{32} \left( \left(4+\tfrac{3i}{2}\right)-(12-26i)\varepsilon -(69+38i)\varepsilon^2  +\dots \right) + \dots \nonumber\\
\delta_3(z)f_3(w) &=& \left( -(1+i)\varepsilon -\left(\tfrac{4}{3}+\tfrac{4i}{3}\right)\varepsilon^3 +\dots \right) \nonumber\\
&\text{ }& -\tfrac{w}{8} \left( (1-i)+(4+2i)\varepsilon +4i\varepsilon^2 +\dots \right)\nonumber\\
&\text{ }& +\tfrac{w^2}{32} \left( 2i-(13+2i)\varepsilon +(8-28i)\varepsilon^2 +\dots \right)+ \dots \nonumber\\ 
\delta_4(z)f_4(w) &=& \left( i\varepsilon -(1+2i)\varepsilon^2 +\left(2-\tfrac{2i}{3}\right)\varepsilon^3 +\dots \right) \nonumber\\
&\text{ }& +\tfrac{w}{8} \left( 1-(4-4i)\varepsilon -(10+14i)\varepsilon^2  +\dots \right)\nonumber\\
&\text{ }& +\tfrac{w^2}{32} \left( \left(1+\tfrac{i}{2}\right)-(6-12i)\varepsilon -(49-26i)\varepsilon^2  +\dots \right) + \dots \nonumber 
\end{eqnarray}
So, near $\lambda_2$ their sum is $-1.$ Near $\lambda_3$ we have
\begin{eqnarray}
\delta_1(z) &=& \tfrac{w}{8}\left(1+2i\varepsilon+\tfrac{8i}{3}\varepsilon^{3}+\dots \right)+ \tfrac{w^2}{32}\left(-\left( 1+ \tfrac{i}{2}\right)+\left(2-2i\right)\varepsilon  \right. \nonumber\\
&\text{ }&- \left. \left(4-2i\right)\varepsilon^{2}+\left(\tfrac{8}{3} - \tfrac{32i}{3} \right) \varepsilon^{3} +\dots\right) +\dots \nonumber\\
\delta_2(z) &=& \tfrac{w}{8}\left( \left(1-i\right)+2\varepsilon +\left(4+2i\right) \varepsilon^{2}+\tfrac{20}{3}\varepsilon^{3}+\dots \right) \nonumber\\
&\text{ }&- \tfrac{w^2}{32}\left(2i-(1-i)\varepsilon -(1-11i)\varepsilon^{2}-\left( \tfrac{16}{3}-\tfrac{16i}{3}\right)\varepsilon^{3}+\dots \right)+\dots \nonumber\\
\delta_3(z) &=& 1-\tfrac{w}{8}\left(3+2i\varepsilon +8\varepsilon^{2}+\tfrac{8i}{3}\varepsilon^{3}+\dots \right)+\tfrac{w^2}{32}\left( \left(2+\tfrac{3i}{2}\right)-(4-4i)\varepsilon \right. \nonumber\\
&\text{ }&+ \left. \left(8+2i\right)\varepsilon^{2}-\left(\tfrac{40}{3}-\tfrac{64i}{3}\right)\varepsilon^{3}+\dots \right)+\dots \nonumber\\
\delta_4(z) &=& \tfrac{w}{8}\left( \left(1-i\right)-2\varepsilon +\left(4-2i\right) \varepsilon^{2}-\tfrac{20}{3}\varepsilon^{3}+\dots \right) \nonumber\\
&\text{ }&- \tfrac{w^2}{32}\left( (1-i)-(1-i)\varepsilon +(5-7i)\varepsilon^{2}-\left( \tfrac{16}{3}-\tfrac{16i}{3}\right) \varepsilon^{3} +\dots \right)+\dots \nonumber
\end{eqnarray}
which gives
\begin{eqnarray}
\delta_1(z)f_1(w) &=& \tfrac{w}{8}\left( 1+2i\varepsilon +\tfrac{8i}{3}\varepsilon^3+\dots \right)\nonumber\\ 
&\text{ }& +\tfrac{w^2}{32} \left( \left(1-\tfrac{3i}{2}\right)+(6+6i)\varepsilon -(12-12i)\varepsilon^2 +\dots \right)+ \dots \nonumber\\
\delta_2(z)f_2(w) &=& -\tfrac{w}{8}\left( (1+i)-(4-2i)\varepsilon -(2+8i)\varepsilon^2  +\dots \right)\nonumber\\
&\text{ }& -\tfrac{w^2}{32}\left( (1+i)-(11-9i)\varepsilon -(33+39i)\varepsilon^2 +\dots \right)+ \dots \nonumber\\
\delta_3(z)f_3(w) &=& -1-\tfrac{w}{8}\left((1-2i)+(4+6i)\varepsilon -8\varepsilon^2  +\dots \right)\nonumber\\
&\text{ }& -\tfrac{w^2}{32}\left( \left(2-\tfrac{5i}{2}\right)+(16+12i)\varepsilon -(24-34i)\varepsilon^2  +\dots \right)+\dots \nonumber\\
\delta_4(z)f_4(w) &=& \tfrac{w}{8}\left( (1-i)+6i\varepsilon -(10+8i)\varepsilon^2 +\dots \right)\nonumber\\
&\text{ }& +\tfrac{w^2}{32} \left( 2-(1-15i)\varepsilon -(45+11i)\varepsilon^2 +\dots \right)+\dots \nonumber
\end{eqnarray}
so their sum is $1.$ Near $\lambda_4$ we have 
\begin{eqnarray}
\delta_1(z) &=& \left((1+i)\varepsilon +\left(\tfrac{4}{3}+\tfrac{4i}{3}\right)\varepsilon^3+\dots \right)\nonumber\\
&\text{ }& +\tfrac{w}{8} \left( (1-i)-2\varepsilon+(4-8i)\varepsilon^2-\tfrac{8}{3}\varepsilon^3\dots \right)\nonumber\\
&\text{ }& -\tfrac{w^2}{32} \left( (1-i)+(2+i)\varepsilon +(8-4i)\varepsilon^2+\left(\tfrac{32}{3}+\tfrac{40i}{3}\right) \varepsilon^3 +\dots \right) +\dots \nonumber\\
\delta_2(z) &=& \left( i\varepsilon +3\varepsilon^2 -\tfrac{8i}{3}\varepsilon^3+\dots \right) +\tfrac{w}{8} \left( 1-4i\varepsilon -\tfrac{40i}{3}\varepsilon^3+\dots \right) \nonumber\\
&\text{ }& -\tfrac{w^2}{32} \left( \left(1+\tfrac{i}{2}\right)+(1-i)\varepsilon +(7+4i)\varepsilon^2 +\left(\tfrac{16}{3} -\tfrac{16i}{3}\right) \varepsilon^3+\dots \right)+...\nonumber\\
\delta_3(z) &=& \left(-(1-i)\varepsilon -\left(\tfrac{4}{3}-\tfrac{4i}{3}\right)\varepsilon^3+\dots \right) \nonumber\\
&\text{ }& + \tfrac{w}{8}\left( (1+i)+2\varepsilon +(4+8i)\varepsilon^2 +\tfrac{8}{3}\varepsilon^3+\dots \right) \nonumber\\
&\text{ }& -\tfrac{w^2}{32} \left( 2-(4-i)\varepsilon -(4-8i)\varepsilon^2-\left(\tfrac{64}{3}+\tfrac{8i}{3}\right)\varepsilon^3+\dots \right) +\dots \nonumber\\
\delta_4(z) &=& \left( 1-3i\varepsilon -3\varepsilon^2+\dots \right)- \tfrac{w}{8} \left( 3-4i\varepsilon +8\varepsilon^2-\tfrac{40i}{3}\varepsilon^3+\dots \right)  \nonumber\\
&\text{ }& +\tfrac{w^2}{32}\left( \left(2+\tfrac{3i}{2}\right)-(1-i)\varepsilon +(11+8i)\varepsilon^2- \left(\tfrac{16}{3}- \tfrac{16i}{3}\right) \varepsilon^3+\dots \right)+\dots \nonumber
\end{eqnarray}
This gives
\begin{eqnarray}
\delta_1(z)f_1(w) &=& \left( (1+i)\varepsilon +\left(\tfrac{4}{3}+\tfrac{4i}{3}\right)\varepsilon^3 +\dots \right) \nonumber\\
&\text{ }& +\tfrac{w}{8} \left( (1-i)+(4+2i)\varepsilon +4i\varepsilon^2 +\dots \right)\nonumber\\
&\text{ }& -\tfrac{w^2}{32} \left( 2i-(13+2i)\varepsilon +(8-28i)\varepsilon^2  +\dots \right)+ \dots \nonumber\\
\delta_2(z)f_2(w) &=& -\left( i\varepsilon -(1+2i)\varepsilon^2 +\left(2-\tfrac{2i}{3}\right)\varepsilon^3 +\dots \right) \nonumber\\
&\text{ }& -\tfrac{w}{8} \left( 1-(4-4i)\varepsilon -(10+14i)\varepsilon^2  +\dots \right)\nonumber\\
&\text{ }& -\tfrac{w^2}{32} \left( \left(1+\tfrac{i}{2}\right)-(6-12i)\varepsilon -(49-26i)\varepsilon^2  +\dots \right) + \dots \nonumber\\ 
\delta_3(z)f_3(w) &=& \left( (1-i)\varepsilon +\left(\tfrac{4}{3}-\tfrac{4i}{3}\right)\varepsilon^3+\dots \right) \nonumber\\
&\text{ }& -\tfrac{w}{8} \left( (1+i)+6i\varepsilon -(8-4i)\varepsilon^2  +\dots \right) \nonumber\\
&\text{ }& -\tfrac{w^2}{32} \left( (3-i)+(7+16i)\varepsilon -(28-16i)\varepsilon^2 +\dots \right) + \dots \nonumber\\
\delta_4(z)f_4(w) &=& \left(1-(2-i)\varepsilon -(1+2i)\varepsilon^2 -\left( \tfrac{2}{3}+\tfrac{2i}{3} \right) \varepsilon^3 +\dots \right)  \nonumber\\
&\text{ }& +\tfrac{w}{8} \left( (1+2i)-(8-8i)\varepsilon -(18+14i)\varepsilon^2  +\dots \right)\nonumber\\
&\text{ }& +\tfrac{w^2}{32} \left( \left(4+\tfrac{3i}{2}\right)-(12-26i)\varepsilon -(69+38i)\varepsilon^2  +\dots \right) + \dots \nonumber
\end{eqnarray}
so, near $\lambda_4$ their sum is $1.$

\subsection{The quadratic polynomial to the $n-$th power }\label{AppendixC4}

As a final step we will compute the expansions for the Riesz projection for the polynomial $p^{n}=(z^{2}-1)^{n}.$ As before we let
$w=z^{2}-1$ so for $|w|<1$ we have
$$z=(w+1)^{1/2}=1+\frac{1}{2}w-\frac{1}{8}w^{2}+\frac{1}{16}w^{3}+\cdots . $$

At the begining of this section we have computed the expansion for $p=z^{2}-1,$ when taking $\varphi \equiv 1 $ near $\lambda_1=1$ and $\varphi \equiv -1$ near $\lambda_2=-1,$ $(\lambda_j,$ $j=1,2,$ being the roots of $p).$ Thus we have,
$ \delta_1(z)=(1+z)/2,\text{ } \delta_2(z) = (1-z)/2$  and
\begin{eqnarray}
f_1(w) &=& 1-\frac{1}{2}w +\frac{3}{8}w^{2}-\frac{5}{16}w^{3} +\dots , \nonumber\\
f_2(w) &=& -1+\frac{1}{2}w -\frac{3}{8}w^{2}+\frac{5}{16}w^{3} +\dots \nonumber
\end{eqnarray}
from which we note that $f_2(w)=-f_1(w).$ 

From section 2 we know that for $\alpha=2 \pi /n$ and a given function $g(w)$
\begin{equation}
 w^k g_k(w^n) =\frac{1}{n} \{ g(w)+ e^{-ik\alpha}g(e^{i\alpha}w) +\dots 
+e^{-i(n-1)k\alpha} g(e^{i(n-1)\alpha} w)\} \label{fk}
\end{equation}
and from Theorem 2.1 
we have
\begin{equation}
\varphi(z) =\sum_{j=1}^d \delta_j(z) [f_{j0}(p(z)^n) + \dots + p(z)^{n-1} f_{j n-1} (p(z)^n)] \label{fzn}
\end{equation}
where $d$ is the degree of $p.$
Denoting
\begin{eqnarray}
F_1(w)&=& f_{10}(w^n) + \dots + w^{n-1} f_{1 n-1} (w^n),\nonumber\\ 
F_2(w)&=& f_{20}(w^n) + \dots + w^{n-1} f_{2 n-1} (w^n),\nonumber
\end{eqnarray}

\noindent since $p(z)=w,$ we have $\displaystyle \varphi(z)=\delta_1(z)F_1(w)+\delta_2(z)F_2(w). $

Since $f_2(w)=-f_1(w)$ we note that also $F_2(w)=-F_1(w).$ Thus it is enough to compute $F_1(w).$ 
We start by replacing $g(w)$ from (\ref{fk}) with $f_1(w)$ and we get
\begin{eqnarray}
 f_{10}(w^n) &=& \tfrac{1}{n} \{ f_1(w)+ f_1(e^{i\alpha}w) +f_1(e^{2i\alpha})+\dots + f_1(e^{i(n-1)\alpha} w)\} \nonumber\\
 f_{11}(w^n) &=& \tfrac{1}{n w} \{ f_1(w)+ e^{-i\alpha}f_1(e^{i\alpha}w) +e^{-2i\alpha}f_1(e^{2i\alpha}w)+\dots \nonumber\\
&\text{ }& + e^{-i(n-1)\alpha} f_1(e^{i(n-1)\alpha} w)\}\nonumber\\
 f_{12}(w^n) &=& \tfrac{1}{n w^2} \{ f_1(w)+ e^{-2i\alpha}f_1(e^{i\alpha}w)+e^{-4i\alpha}f_1(e^{2i\alpha}w) +\dots \nonumber\\
&\text{ }& + e^{-2i(n-1)\alpha} f_1(e^{i(n-1)\alpha} w)\}\nonumber\\
\dots \nonumber\\
 f_{1 n-1}(w^n) &=& \tfrac{1}{n w^{n-1}} \{ f_1(w)+ e^{-i(n-1)\alpha}f_1(e^{i\alpha}w)+e^{-2i(n-1)\alpha}f_1(e^{2i\alpha}w) +\dots \nonumber\\
&\text{ }& + e^{-i(n-1)^{2}\alpha} f_1(e^{i(n-1)\alpha} w)\}\nonumber .
\end{eqnarray} 
Then
\begin{eqnarray}
F_1(w) &=& f_{10}(w^{n})+w f_{11}(w^{n})+w^{2}f_{12}(w^{n})+\dots +w^{n-1}f_{1 n-1}(w^{n}) \nonumber\\
&=& \frac{1}{n} \left( nf_1(w)+ f_1(e^{i\alpha}w)\sum_{k=0}^{n-1} e^{-ik\alpha} +f_1(e^{2i\alpha}w) \sum_{k=0}^{n-1}e^{-2ki\alpha} \right. \nonumber
\end{eqnarray}
\begin{eqnarray}
&\text{ }& \left. +f_1(e^{3i\alpha}w) \sum_{k=0}^{n-1} e^{-3ki\alpha} + \dots +f_1 (e^{i(n-1)\alpha}w) \sum_{k=0}^{n-1} e^{-ik(n-1)\alpha}\right)   \nonumber\\
&=& \frac{1}{n} \sum_{j=0}^{n-1} \left( f_1(e^{ij\alpha}w)\sum_{k=0}^{n-1}e^{-ijk\alpha}\right) \label{F1}
\end{eqnarray}
Now we must check that near $\lambda_1,$ the quantity $\varphi(z)=\delta_1(z)F_1(w)+\delta_2(z)F_2(w)$ is identically $1$ and near the other root is $-1.$ Since $f_1(w)= z^{-1} =((w+1)^{1/2})^{-1}$ it follows that $f_1(e^{ij\alpha}w)= (e^{ij\alpha}w+1)^{-1/2}$ which has the expansion
$$f_1(e^{ij\alpha}w)= 1-\frac{1}{2}e^{ij\alpha}w +\frac{3}{8}e^{2ij\alpha}w^{2}-\frac{5}{16}e^{3ij\alpha}w^{3}+\dots $$
\noindent Replacing this in (\ref{F1}) we get
\begin{eqnarray}
F_1(w) &=& \frac{1}{n} \sum_{j=0}^{n-1}\left( \left(  1-\frac{1}{2}e^{ij\alpha}w +\frac{3}{8}e^{2ij\alpha}w^{2}-\frac{5}{16}e^{3ij\alpha}w^{3}+\dots \right) \sum_{k=0}^{n-1}e^{-ijk\alpha}\right) \nonumber\\
&=& \frac{1}{n} \sum_{j=0}^{n-1}\sum_{k=0}^{n-1}e^{-ijk\alpha} - \frac{1}{2n}w \sum_{j=0}^{n-1}\left( e^{ij\alpha} \sum_{k=0}^{n-1}e^{-ijk\alpha}\right) \nonumber\\
&\text{ }& + \frac{3}{8n}w^{2} \sum_{j=0}^{n-1}\left( e^{2ij\alpha} \sum_{k=0}^{n-1}e^{-ijk\alpha}\right)-\frac{5}{16 n}w^{3} \sum_{j=0}^{n-1}\left( e^{3ij\alpha} \sum_{k=0}^{n-1}e^{-ijk\alpha}\right)+\dots \nonumber\\
&=& 1-\frac{1}{2}w +\frac{3}{8}w^{2}-\frac{5}{16}w^{3} +\dots \nonumber
\end{eqnarray}
since \\ \\
$ \displaystyle\sum_{j=0}^{n-1}\sum_{k=0}^{n-1}e^{-ijk\alpha}=n, \text{ }
 \displaystyle\sum_{j=0}^{n-1}\left( e^{ij\alpha} \sum_{k=0}^{n-1} e^{-ijk\alpha}\right) = n,\text{ }
 \displaystyle\sum_{j=0}^{n-1}\left( e^{2ij\alpha} \sum_{k=0}^{n-1}e^{-ijk\alpha}\right)=n, 
$ $ \displaystyle \sum_{j=0}^{n-1}\left( e^{3ij\alpha} \sum_{k=0}^{n-1}e^{-ijk\alpha}\right)=n$
and so on. Similarly we get $F_2(w)=f_2(w).$\\ \\ Hence, near $\lambda_1 $ we have 
\begin{eqnarray}
\delta_1(z)F_1(w) &=& 1-\frac{1}{4}w +\frac{3}{16}w^{2} -\frac{5}{32}w^{3}+\dots ,\nonumber\\
\delta_2(z)F_2(w) &=& \frac{1}{4}w -\frac{3}{16}w^{2} +\frac{5}{32}w^{3} +\dots \nonumber
\end{eqnarray}

\noindent so their sum is identically $1.$ Near $-1$ we have 
\begin{eqnarray}
\delta_1(z)F_1(w) &=& -\frac{1}{4}w +\frac{3}{16}w^{2} -\frac{5}{32}w^{3}+\dots ,\nonumber\\
\delta_2(z)F_2(w) &=& -1+\frac{1}{4}w -\frac{3}{16}w^{2} +\frac{5}{32}w^{3} +\dots .\nonumber
\end{eqnarray}

\noindent So, their sum is $-1.$

\section{Separating polynomials} \label{AppendixD}

\hspace{4mm} Below, one can see how monic polynomials of degrees $6,8,10,12$ and $14$ get separated and how they look like when appling the perturbations.

For $p(z)=z^{6}-1$ we perturb only the four complex roots and we leave unchanged the other two real roots,
\begin{eqnarray}
p(z) &=& z^{6}-1=(z^{2}-1)(z^{4}+z^{2}+1)\nonumber\\
&=& (z^{2}-1) (z-e^{i\theta})(z-e^{2i\theta})(z+e^{i\theta})(z+e^{2i\theta})\nonumber\\
&=& (z^{2}-1)(z^{2}-e^{2i\theta})(z^{2}-e^{4i\theta})
\end{eqnarray}
where $\theta=\pi/3.$ 

Let $\theta_\varepsilon=\theta-\varepsilon,$ then $l:|p_{\varepsilon}(z)|=1-\eta.$ Letting $\varepsilon=\pi/70,$ we have

\begin{figure}[hb]
	\centering
	\begin{subfigure}[b]{0.48\textwidth}
		\includegraphics[width=\textwidth]{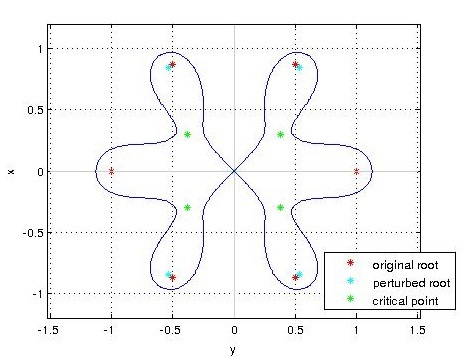}
		\caption{Level 1}
        \label{fig1a}
    \end{subfigure}
	\begin{subfigure}[b]{0.48\textwidth}
		\includegraphics[width=\textwidth]{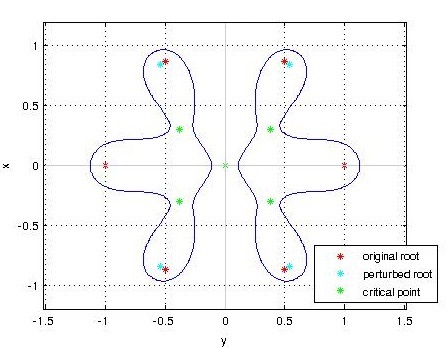}
		\caption{Level 0.998}
        \label{fig1b}
    \end{subfigure}
	\caption{$p(z)=z^{6}-1$}
	\label{fig1}
\end{figure}
\vspace{2cm}
For $p(z)=z^{8}-1$ we first need to apply a rotation with $\pi/8$ so that no root lays on the imaginary axis. Thus our polynomial becomes $$p(z)=z^{8}+1 $$ with roots $e^{i\pi/8},$ $e^{3i\pi/8},$ $e^{5i\pi/8},$ $e^{7i\pi/8},$ $e^{9i\pi/8},$ $e^{11i\pi/8},$ $e^{13i\pi/8}$ and $e^{15i\pi/8}.$ 

We perturbed only the four roots closest to the imaginary axis with $\varepsilon.$

Therefore the perturbed polynomial is
\begin{eqnarray}
p_\varepsilon(z) &=& (z-e^{i\pi /8})(z-e^{i(3\pi /8-\varepsilon)})(z-e^{i(5\pi /8+\varepsilon)})(z-e^{7i\pi /8})\nonumber\\
&\text{ }& (z-e^{9i\pi /8})(z-e^{i(11\pi /8-\varepsilon)})(z-e^{i(13\pi /8+\varepsilon)})(z-e^{15i\pi /8}) \nonumber
\end{eqnarray}
and for this one we compute the lemniscate and we plot it.
\newpage
\begin{figure}[hb]
	\centering
	\begin{subfigure}[b]{0.48\textwidth}
		\includegraphics[width=\textwidth]{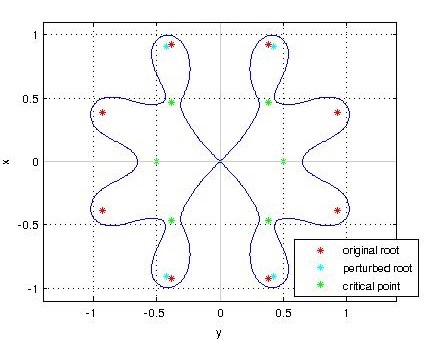}
		\caption{Level 1}
        \label{fig2a}
    \end{subfigure}
	\begin{subfigure}[b]{0.48\textwidth}
		\includegraphics[width=\textwidth]{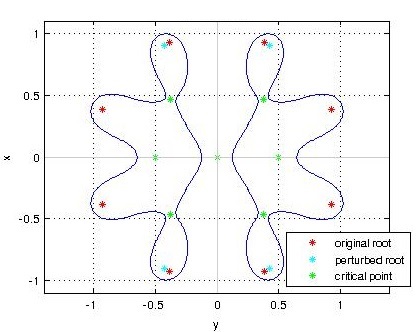}
		\caption{Level 0.998}
        \label{fig2b}
    \end{subfigure}
	\caption{$p(z)=z^{8}-1$}
	\label{fig2}
\end{figure}
\vspace{2cm}

For $p(z)=z^{10}-1$ with roots $1,$ $-1,$ $e^{i\pi/5},$ $e^{2i\pi/5},$ $e^{3i\pi/5},$ $e^{4i\pi/5},$ $e^{6i\pi/5},$ $e^{7i\pi/5},$ $e^{8i\pi/5}$ and $e^{9i\pi/5}$ we perturb the four roots that are closest to $i\mathbb{R}$ and we get
\begin{eqnarray}
p_\varepsilon(z) &=& (z^{2}-1)(z-e^{i\pi /5})(z-e^{i(2\pi /5-\varepsilon)})(z-e^{i(3\pi /5+\varepsilon)})(z-e^{4i\pi /5})\nonumber\\
&\text{ }& (z-e^{6i\pi /5})(z-e^{i(7\pi /5-\varepsilon)})(z-e^{i(8\pi /5+\varepsilon)})(z-e^{9i\pi /5}). \nonumber
\end{eqnarray}
For $p_\varepsilon(z)$ we compute the lemniscate and we get the following picture:

\begin{figure}[hb]
	\centering
	\begin{subfigure}[b]{0.49\textwidth}
		\includegraphics[width=\textwidth]{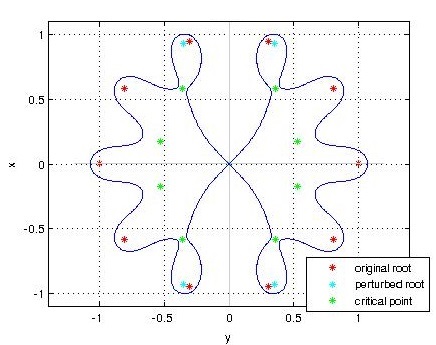}
		\caption{Level 1}
        \label{fig3a}
    \end{subfigure}
	\begin{subfigure}[b]{0.49\textwidth}
		\includegraphics[width=\textwidth]{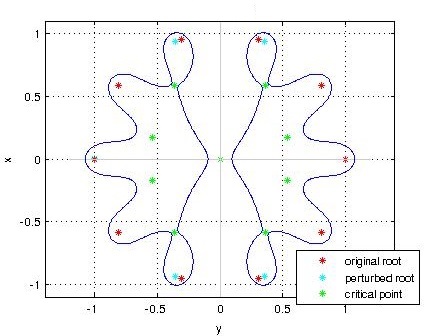}
		\caption{Level 0.999}
        \label{fig3b}
    \end{subfigure}
	\caption{$p(z)=z^{10}-1$}
	\label{fig3}
\end{figure}

\vspace{2cm}
For $p(z)=z^{12}-1$ we first apply a rotation with $\pi/12$ and we get a new polynomial $p(z)=z^{12}+1$ to which we apply the perturbation with $\varepsilon.$ In this case the results can be seen in Figure \ref{fig4}:
\newpage
\begin{figure}[hb]
	\centering
	\begin{subfigure}[b]{0.49\textwidth}
		\includegraphics[width=\textwidth]{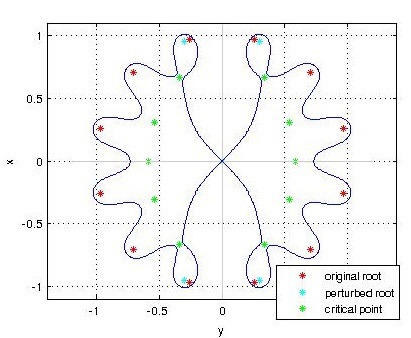}
		\caption{Level 1}
        \label{fig4a}
    \end{subfigure}
	\begin{subfigure}[b]{0.49\textwidth}
		\includegraphics[width=\textwidth]{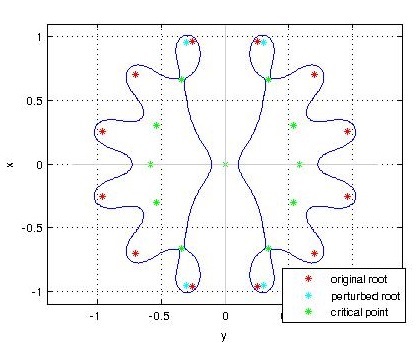}
		\caption{Level 0.999}
        \label{fig4b}
    \end{subfigure}
	\caption{$p(z)=z^{12}-1$}
	\label{fig4}
\end{figure}
\vspace{2cm}
And the last case that we discuss here is $p(z)=z^{14}-1.$ The roots of this polynomial do not need any rotation, so we just change the $4$ roots that are closest to $i\mathbb{R}$ with  $\varepsilon$ and then we compute and plot the lemniscate:

\begin{figure}[hb]
	\centering
	\begin{subfigure}[b]{0.49\textwidth}
		\includegraphics[width=\textwidth]{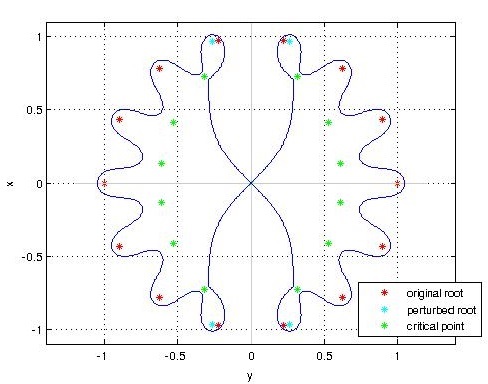}
		\caption{Level 1}
        \label{fig5a}
    \end{subfigure}
	\begin{subfigure}[b]{0.49\textwidth}
		\includegraphics[width=\textwidth]{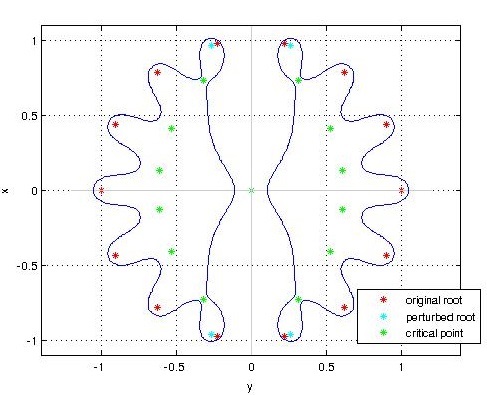}
		\caption{Level 0.999}
        \label{fig5b}
    \end{subfigure}
	\caption{$p(z)=z^{14}-1$}
	\label{fig5}
\end{figure}
 
\section{Pictures with lowest level of the lemniscate that holds the separation} \label{AppendixE}

The following pictures are what we have riched when perturbing the roots with $\varepsilon=\pi/70$ and decreasing the level to its lowest value that ensures the separation into two parts, each on one side of the imaginary axis :

\begin{figure}[hb]
    \centering
    \begin{subfigure}[b]{0.48\textwidth}
        \includegraphics[width=\textwidth]{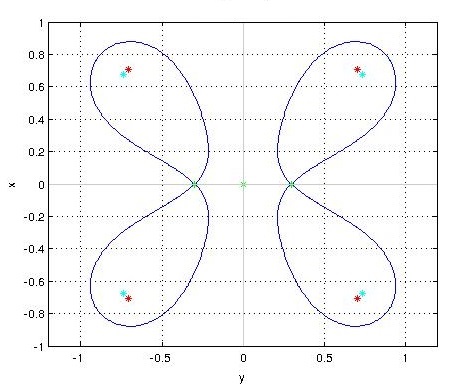}
        \caption{$p(z)=z^{4}-1$}
        \label{fig6a}
        \end{subfigure}
    \begin{subfigure}[b]{0.49\textwidth}
        \includegraphics[width=\textwidth]{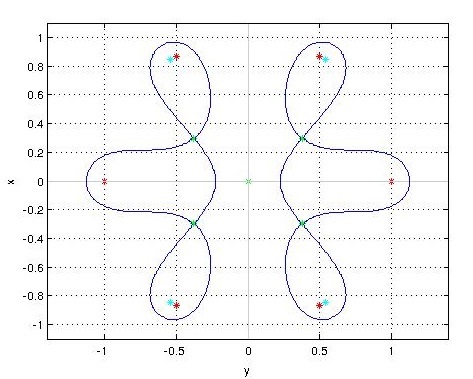}
        \caption{$p(z)=z^{6}-1$}
        \label{fig6b}
    \end{subfigure}
    \begin{subfigure}[b]{0.49\textwidth}
        \includegraphics[width=\textwidth]{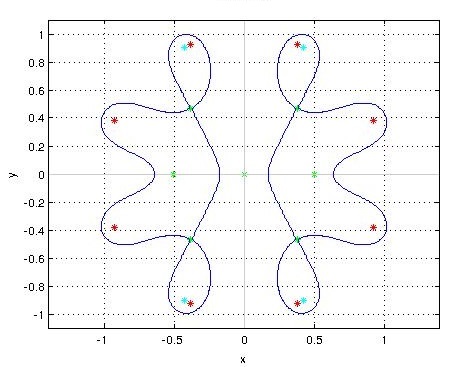}
        \caption{$p(z)=z^{8}-1$}
        \label{fig6c}
    \end{subfigure}
    \begin{subfigure}[b]{0.49\textwidth}
        \includegraphics[width=\textwidth]{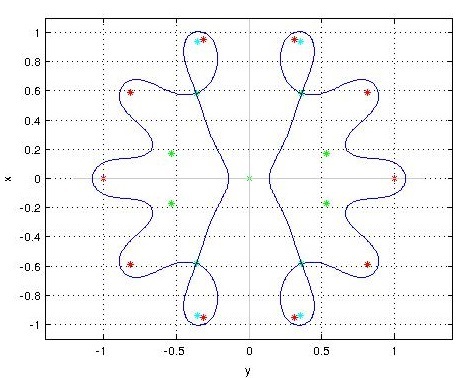}
        \caption{$p(z)=z^{10}-1$}
        \label{fig6d}   
    \end{subfigure}
    \begin{subfigure}[b]{0.49\textwidth}
        \includegraphics[width=\textwidth]{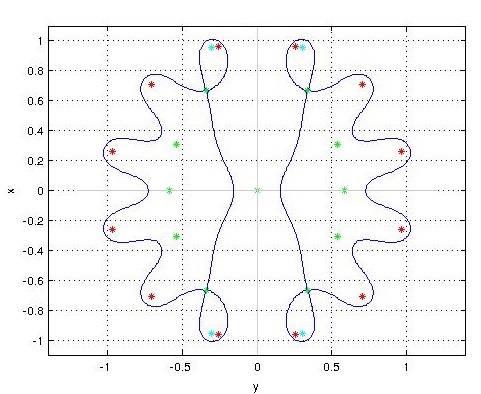}
        \caption{$p(z)=z^{12}-1$}
        \label{fig6e}
    \end{subfigure}
    \begin{subfigure}[b]{0.49\textwidth}
        \includegraphics[width=\textwidth]{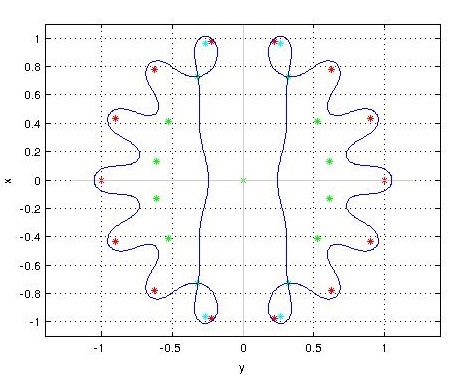}
        \caption{$p(z)=z^{14}-1$}
        \label{fig6f}
    \end{subfigure}
    \caption{Separation with lowest level for $\varepsilon=\pi/70 $}\label{fig6}
\end{figure}

\end{appendices}

\end{document}